\documentclass[10pt]{article}
\usepackage{amsfonts}
\usepackage{amssymb,amsmath,amsthm, amsfonts}
\usepackage{amsmath}

\allowdisplaybreaks[4]

\def\sqr#1#2{{\vcenter{\vbox{\hrule height.#2pt
              \hbox{\vrule width.#2pt height#1pt \kern#1pt \vrule width.#2pt}
          \hrule height.#2pt}}}}
\def\sqr#1#2{{\vcenter{\vbox{\hrule height.#2pt
              \hbox{\vrule width.#2pt height#1pt \kern#1pt \vrule width.#2pt}
              \hrule height.#2pt}}}}
\def\3n{\negthinspace \negthinspace \negthinspace }
\def\2n{\negthinspace \negthinspace }
\def\1n{\negthinspace }

\def\={\buildrel \triangle \over =}

\def\exp{\mathop{\rm exp}}
\def\sup{\mathop{\rm sup}}
\def\inf{\mathop{\rm inf}}

\def\inf{\hbox{\rm inf$\,$}}

\def\sup{\hbox{\rm sup}}
\def\inf{\hbox{\rm inf}}

\def\|{\Big |}
\def\({\Big (}
\def\){\Big )}
\def\[{\Big[}
\def\]{\Big]}
\def\be{\begin{equation}}
\def\bel{\begin{equation}\label}
\def\ee{\end{equation}}
\def\bt{\begin{theorem}}
\def\bcd{\begin{condition}}
\def\ecd{\end{condition}}
\def\et{\end{theorem}}
\def\bc{\begin{corollary}}
\def\ec{\end{corollary}}
\def\bde{\begin{definition}}
\def\ede{\end{definition}}
\def\bl{\begin{lemma}}
\def\el{\end{lemma}}
\def\bp{\begin{proposition}}
\def\ep{\end{proposition}}
\def\bex{\begin{example}}
\def\eex{\end{example}}
\def\br{\begin{remark}}
\def\er{\end{remark}}
\def\ba{\begin{array}}
\def\ea{\end{array}}
\def\ed{\end{document}}

\def\square#1{\vbox{\hrule\hbox{\vrule height#1%
     \kern#1\vrule}\hrule}}
\def\rectangle#1#2{\vbox{\hrule\hbox{\vrule height#1%
     \kern#2\vrule}\hrule}}

\font\tenbb=msbm10 \font\sevenbb=msbm7 \font\fivebb=msbm5

\newfam\bbfam
\scriptscriptfont\bbfam=\fivebb \textfont\bbfam=\tenbb
\scriptfont\bbfam=\sevenbb

\newtheorem{lemma}{Lemma}[section]
\newtheorem{remark}{Remark}[section]
\newtheorem{example}{Example}[section]
\newtheorem{theorem}{Theorem}[section]
\newtheorem{corollary}{Corollary}[section]

\newtheorem{definition}{Definition}[section]
\newtheorem{proposition}{Proposition}[section]
\newtheorem{condition}{Condition}[section]

\makeatletter
   
   \@addtoreset{equation}{section}
\makeatother

\begin{document}

\title{Fully coupled mean-field FBSDEs with jumps and related optimal control problems}
\author{ Wenqiang Li$^1$\footnote{W. Li acknowledges the financial support partly by Natural Science Foundation of Shandong Province (ZR2017MA015), Doctoral Scientific Research  fund of Yantai University (No. SX17B09), A Project of Shandong  Province Higher Educational Science and Technology Program (No. J17KA162).} and  Hui Min$^2$\footnote{Corresponding author. H. Min acknowledges the financial support partly by Beijing Natural Science Foundation (No. 1184013), China Postdoctoral Science Foundation Funded Project (No. 2016M600020), Beijing Postdoctoral Research Foundation, Chaoyang Postdoctoral Research Foundation. }\\
{\small $^1$School of Mathematics and Information Sciences, Yantai University, Yantai, P.~R.~China.}\\
   {\small $^2$College of Applied Sciences,
        Beijing University of Technology, Beijing, P.~R.~China.}\\
{\small {\it E-mails: wenqiangli@ytu.edu.cn; huimin@bjut.edu.cn.}}\\
}

\date{November 1, 2018}
\maketitle
\noindent{\bf{Abstract}}

This paper study a type of fully coupled mean-field forward-backward stochastic differential equations  with jumps under the monotonicity condition, including the existence and the uniqueness of the solution of our equation as well as the continuity property of the solutions with respect to the parameters. Then we establish the stochastic maximum principle for the corresponding optimal control problems and give the applications to mean-variance portfolio problems and linear-quadratic problems, respectively.

 %Then we discuss the stochastic optimal control problems of mean-field FBSDEs. The stochastic maximum principles are derived and the related mean-field linear quadratic optimal control problems are also discussed.

\noindent{{\bf AMS Subject classification:} 60H10; 60H30; 35K65

\noindent{{\bf Keywords:} Mean-field backward stochastic differential equation with jumps; fully coupled forward-backward stochastic differential equation; monotonicity conditions; stochastic maximum principle; mean-variance protfolio problem; linear-quadratic problem.

\section{Introduction}

Forward-backward stochastic differential equations (FBSDEs, for short) have attracted significant attention
because of their wide range of applications, from solving nonlinear partial differential equations (PDEs, for short), pricing American options to describing some optimization problems (refer to, \cite{MY}). Inspired by the introduction of a recursive stochastic utility function in \cite{DE1}, Antonelli \cite{A}  first investigated the existence and the uniqueness of the solution of FBSDEs driven by Brownian motion with requiring the small enough Lipschitz constant of the coefficients. In order to deal with fully coupled FBSDEs on an arbitrarily given time interval, Ma, Protter, Yong \cite{MPY} introduced a ``four-step scheme" approach which combines probability methods and PDE methods. Using this method, they obtained the existence and the uniqueness of the solution with deterministic and non-degenerate diffusion coefficients. Peng and Wu \cite{PW}
used a purely probabilistic continuation method to study fully coupled FBSDEs with additional monotonicity condition on the coefficients. There are also many other methods to study the solution of FBSDEs, see Delarue \cite{D} and Zhang \cite{Z} for numerical approaches, Ma, et al. \cite{MWZZ} for a unified approach, etc. For more details about fully coupled FBSDEs, the readers also refer to Ma and Yong \cite{MY}, or  Yong \cite{Y} and the references therein.

On the other hand,  mean-field limits are widely applied to many diverse areas such as statistical physics,
quantum mechanics and quantum chemistry. Based on this, Buckdahn, et al. \cite{BLP2} obtained a new type of BSDEs, namely mean-field BSDEs. In \cite{BLP1}, Buckdahn, Li and Peng made an in-depth study of such type of BSDEs and got the existence and the uniqueness of the solution of mean-field BSDEs, as well as a comparison theorem. They also established the link between the solution of this mean-field BSDEs and some nonlocal PDEs.
Min, Peng, Qin \cite{MPQ} generalized their work to fully coupled mean-field FBSDEs cases. Barles, Buckdahn, Pardoux \cite{BBP} studied a new type of BSDEs driven by Brownian motion and a Poisson random measure, namely BSDEs with jumps and showed the connection with a system of parabolic integro-PDEs. Royal \cite{Royer} gave a strict comparison theorem for BSDEs with jumps and the relation to non-linear expectation.
Li, Min \cite{LM} investigated a new type of mean-field BSDEs with jumps, namely mean-field BSDEs with jumps involving value function and obtained the related dynamic programming principle.

Inspired by the above works, one of our aim is to study a type of fully coupled mean-field FBSDEs with jumps.
 To the best of our knowledge, no corresponding works have been done until now. To be more specific, we consider the following fully coupled mean-field FBSDEs with jumps:
\begin{equation}\label{Eq1.1}
\left\{
\begin{aligned}
dx(t)=&\ \int_EE'\[b\big(t,\lambda(t,e),(\lambda(t,e))'\big)\]\lambda(de)dt
 +\int_EE'\[\sigma\big(t,\lambda(t,e),(\lambda(t,e))'\big)\]\lambda(de)dB_t\\
&\ +\int_EE'\[h\big(t,\lambda(t-,e),(\lambda(t-,e))',e\big)\]\widetilde{\mu}(dt,de),\\
-dy(t)=&\ \int_EE'\[f\big(t,\lambda(t,e),(\lambda(t,e))'\big)\]\lambda(de)dt
 -z(t)dB_t-\int_Ek(t,e)\widetilde{\mu}(dt,de),\, t\in[0,T],\\
x(0)=&\ a,\\
y(T)=&\ E'[\Phi(x(T),(x(T))')],
\end{aligned}
\right.
\end{equation}
where $$\lambda(t,e)=\big(x(t),y(t),z(t),k(t,e)\big),\ \lambda(t-,e)=\big(x(t-),y(t-),z(t),k(t,e)\big),$$
%where $(x(.),y(.),z(.),k(.))$ take values in $ {\mathbf{R}}^n \times {\mathbf{R}}^m \times {\mathbf{R}}^{m\times d}\times\mathbb{R}^m$;
$b, \sigma,h, f, \Phi$ are mappings with appropriate dimensions, $T\geq 0$ is an arbitrarily fixed number. Under the classical assumption (H3.1) and monotonicity assumption (H3.2), the existence and the uniqueness of the solution of our fully coupled mean-field FBSDEs with jumps (\ref{Eq1.1}) are obtained
by using a purely probabilistic continuation method (See, Theorem 3.1). Furthermore, we study the continuity of the solution of equation (\ref{Eq1.1}) relying on parameters under our assumptions (See, Theorem 4.1).

Another aim of this paper is to study the related optimal control problems for the controlled fully coupled mean-field FBSDEs with jumps (\ref{Eq1.1}) in a Markovian framework. Our motivation of this part is followed from many theoretical works and a wide range of applications with respect to the stochastic maximum principle of the stochastic control problems under jump-diffusion framework. Framstad, Oksendal, Sulem \cite{FOS} proved a sufficient maximum principle for the optimal control of jump diffusions and gave applications to optimization problems in a financial market. Oksendal, Sulem \cite{OS2009} and Shi, Wu \cite{SW}  studied  the maximum principle for optimal control of FBSDEs with jumps   by using different approaches, respectively. Shen, Siu \cite{SS} generalized their work to mean-field cases.

Let us give the specific description of our control problem, where the dynamic has the following form
\begin{equation}\label{Eq2018102401}
\left\{
\begin{aligned}
dx^v(t)=&\ \int_EE'[b\big(t,\pi^v(t,e),v(t)\big)]\lambda(de)dt +\int_EE'[\sigma\big(t,\pi^v(t,e),v(t)\big)]\lambda(de)dB_t\\
&\ +\int_EE'[h\big(t,\pi^v(t-,e),v(t),e\big)]\widetilde{\mu}(dtde),\\
-dy^v(t)=&\ \int_EE'[f\big(t,\pi^v(t,e),v(t)\big)]\lambda(de)dt-z^v(t)dB_t-\int_Ek^v(t,e)\widetilde{\mu}(dtde),\\
x^v(0)=&\ a,\ \ y^v(T)=E'[\Phi\big(x^v(T),(x^v(T))'\big)],
\end{aligned}
\right.
\end{equation}
where $$\pi^v(t,e)=\big(x^v(t), y^v(t), z^v(t), k^v(t,e), (x^v(t))',(y^v(t))',(z^v(t))', (k^v(t,e))'\big),$$
$$\pi^v(t-,e)=\big(x^v(t-), y^v(t-), z^v(t), k^v(t,e), (x^v(t-))',(y^v(t-))', (z^v(t))', (k^v(t,e))'\big),$$
and the cost functional has the following form
\begin{equation}\label{equ 2018102402}
\begin{aligned}
 J(v(\cdot))=E\Big[ & \displaystyle \int_0^T\int_EE'\big[g\big(t,\pi^v(t,e),v(t)\big)\big]\lambda(de)dt+E'[\varphi\big(x^v(T),(x^v(T))'\big)]+\gamma\big(y^v(0)\big)\Big],
\end{aligned}
\end{equation}
where all coefficients of the dynamic and the cost functional are given deterministic functions (See, Section 5 for more details). Our control domain is convex and we get the necessary  and sufficient condition for the optimality of the control with the help of a convex perturbation (See, Theorem 5.1 and 5.2). Moreover, we apply these results to a mean-variance portfolio selection mixed with a mean-field recursive utility and a linear-quadratic optimal control problem, respectively.

The rest of the paper is organized as follows. In Section 2, we introduce the framework of our study and some results on mean-field forward and backward SDEs with jumps. In Section 3, we prove the existence and the uniqueness of solution of fully coupled mean-field FBSDEs with jumps.We present
the continuity of solutions of our equation with respect to the parameters in Section 4. Section 5 is devoted to discussing the necessary and sufficient condition of the optimal control problem for the related fully coupled mean-field FBSDEs with jumps. In Section 6 we give two applications to illustrate the results of Section 5. A corresponding lemma (used in the proof of Theorem 3.1) and its proof are given in Appendix.

\section{Preliminaries}

 Let $(\Omega,\mathcal{F},P)$ be a probability space which is the completed product of the Wiener space $(\Omega_1,\mathcal{F}_1,P_1)$ and the Poisson space $(\Omega_2,\mathcal{F}_2,P_2)$:\\
$\bullet \,\,(\Omega_1,\mathcal{F}_1,P_1)$ is a classical Wiener space, where $\Omega_1=C_0(\mathbb{R};\mathbb{R}^d)$ is the set of continuous functions from $\mathbb{R}$ to $\mathbb{R}^d$ with value 0 in time 0, $\mathcal{F}_1$ is the completed Borel $\sigma$-algebra over $\Omega_1$, and $P_1$ is the Wiener measure such that $B_s(\omega)=\omega_s$, $s\in\mathbb{R}_+$, $\omega\in\Omega_1$, and $B_{-s}(\omega)=\omega(-s)$, $s\in\mathbb{R}_+$, $\omega\in\Omega_1$, are two independent $d$-dimensional Brownian motions. The natural filtration $\{\mathcal{F}_s^B,s\geq0\}$ is generated by $\{B_s\}_{s\geq0}$ and augmented by all $P_1$-null sets, i.e.,
$$\mathcal{F}_s^B=\sigma\{B_r,r\in(-\infty,s]\}\vee\mathcal{N}_{P_1},\,\,s\geq0.$$
$\bullet \,\,(\Omega_2,\mathcal{F}_2,P_2)$ is a Poisson space. We denote by $p:\,D_p\subset\mathbb{R}\rightarrow E$ the point functions, where $D_p$ is a countable subset of the real line $\mathbb{R}$, $E=\mathbb{R}^l\setminus\{0\}$ is equipped with its Borel $\sigma$-field $\mathcal{B}(E)$. We introduce the counting measure $\mu(p,dtde)$ on $\mathbb{R}\times E$ as follows:
$$\mu(p,(s,t]\times\Delta)=\sharp\{r\in D_p\cap(s,t]:p(r)\in\Delta\},\,\Delta\in\mathcal{B}(E),\,s,t\in\mathbb{R},\,s<t,$$
where $\sharp$ denotes the cardinal number of the set. We identify the point function $p$ with $\mu(p,\cdot)$. Let $\Omega_2$ be the set of all point functions $p$ on $E$, and $\mathcal{F}_2$ be the smallest $\sigma$-field on $\Omega_2$. The coordinate mappings $p\rightarrow\mu(p,(s,t]\times\Delta)$, $s,t\in\mathbb{R}$, $s<t$, $\Delta\in\mathcal{B}(E)$, are measurable with respect to $\mathcal{F}_2$. On the measurable space $(\Omega_2,\mathcal{F}_2)$ we consider the probability measure $P_2$ such that the canonical coordinate measure $\mu(p,dtde)$ becomes a Poisson random measure with the compensator $\widehat{\mu}(dtde)=dt\lambda(de)$; the process $\{\widetilde{\mu}((s,t]\times A)=(\mu-\widehat{\mu})((s,t]\times A)\}_{s\leq t}$ is a martingale, for any $A\in\mathcal{B}(E)$ satisfying $\lambda(A)<\infty$. Here $\lambda$ is supposed to be a $\sigma$-finite measure on $(E,\mathcal{B}(E))$ with $\int_E(1\wedge|e|^2)\lambda(de)<\infty$. The filtration $\{\mathcal{F}_t^\mu\}_{t\geq0}$ generated by the coordinate measure $\mu$ is introduced by setting:
$$\dot{\mathcal{F}}_t^\mu=\sigma\{\mu((s,r]\times \Delta):-\infty<s\leq r\leq t,\Delta\in\mathcal{B}(E)\},\,t\geq0,$$
and taking the right-limits $\mathcal{F}_t^\mu=(\bigcap\limits_{s>t}\dot{\mathcal{F}}_s^\mu)\vee\mathcal{N}_{P_2},\,t\geq0$, augmented by all the $P_2$-null sets. At last, we set $(\Omega,\mathcal{F},P)=(\Omega_1\times \Omega_2,\mathcal{F}_1\otimes\mathcal{F}_2,P_1\otimes P_2)$, where $\mathcal{F}$ is completed with respect to $P$, and the filtration $\mathbb{F}=\{\mathcal{F}_t\}_{t\geq0}$ is generated by
$$\mathcal{F}_t:=\mathcal{F}_t^B\otimes\mathcal{F}_t^\mu,\,t\geq0,\,\mbox{augmented\,\,by\,\,all\,\,P-null\,\,sets}.$$
\indent For any $n\geq1$, $|z|$ denotes the Euclidean norm of $z\in\mathbb{R}^n$. Fix $T>0$, we also shall introduce the following three spaces of processes which will be used frequently in what follows:\\

\indent$\mathcal{S}_{\mathbb{F}}^2(0,T;\mathbb{R})$ := $\{(\psi_t)_{0\leq t\leq T}$ real-valued $\mathbb{F}$-adapted c\`{a}dl\`{a}g process : $E[\sup_{0\leq t\leq T}|\psi_t|^2]<+\infty\}$;\\

\indent$\mathcal{H}_{\mathbb{F}}^2(0,T;\mathbb{R}^n)$ := $\{(\psi_t)_{0\leq t\leq T}$ $\mathbb{R}^n$-valued $\mathbb{F}$-progressively measurable process : \\ \indent\indent\qquad\qquad$||\psi||^2=E[\int_0^T|\psi_t|^2dt]<+\infty\}$;\\

\indent$\mathcal{K}_{\mathbb{F},\lambda}^2(0,T;\mathbb{R}^n)$ := $\{K: \Omega\times[0,T]\times E \rightarrow \mathbb{R}^n$ $\mathcal{P}\otimes\mathcal{B}(E)$-measurable mapping :\\
\indent\indent\qquad\qquad $|K|_{L^2(\lambda)}^2=E[\int_0^T\int_E|K_t(e)|^2\lambda(de)dt]<+\infty\}$.\footnote{$\mathcal{P}$ denotes the $\sigma$-algebra of $\mathcal{F}_t$-predictable sub-sets of $\Omega\times[0,T]$.}\\

\indent For the reader's convenience, let us first introduce the framework of mean-field SDEs with jumps and mean-field BSDEs  with jumps which will be used in the follows. For more details we refer to  \cite{LM}.

\indent Let $(\bar{\Omega}, \bar{\mathcal{F}},\bar{P})=(\Omega',\mathcal{F}',P')\otimes(\Omega,\mathcal{F},P)=(\Omega,\mathcal{F},P)\otimes(\Omega,\mathcal{F},P)$ be the (non-completed) product of $(\Omega, \mathcal{F},P)$  with itself. Let us endow the product space $(\bar{\Omega}, \bar{\mathcal{F}},\bar{P})$ with the filtration $\bar{\mathbb{F}}=\{ {\bar{\mathcal{F}}}_t=\mathcal{F} \otimes \mathcal{F}_t, 0 \leq t \leq T\}$.\\
\indent Given a random variable $\xi$ over $(\Omega,\mathcal{F},P)$, we denote by $\xi'$ its (under $\bar{P}$) independent copy on $(\Omega',\mathcal{F}',P')$: $ \xi'(\omega)=\xi(\omega),\,\omega\in\Omega'(=\Omega)$. Extending $\xi,\,\xi'$ canonically to $\bar{\Omega}$, $\xi(\omega',\omega)=\xi(\omega)$, $\xi'(\omega',\omega)=\xi'(\omega')$, $(\omega,\omega')\in\bar{\Omega}=\Omega'\times\Omega$, we have for all nonnegative Borel functions $f: \mathbb{R}^2\rightarrow\mathbb{R}_+$, $E'[f(\xi',\xi)]=\int_{\Omega'}f(\xi'(\omega'),\xi)P'(d\omega')=E[f(\xi,x)]|_{x=\xi}$.\\
\indent The driving coefficient of our mean-field BSDE with jumps is a mapping
\begin{equation}\nonumber
f=f(\bar{\omega},t,y,z,k,y',z',k'): \bar{\Omega}\times[0,T]\times\mathbb{R}^m\times\mathbb{R}^{m\times d}\times L^2(E,\mathcal{B}(E),\lambda;\mathbb{R}^m)\times\mathbb{R}^m\times\mathbb{R}^{m\times d}\times L^2(E,\mathcal{B}(E),\lambda;\mathbb{R}^m)\rightarrow\mathbb{R}^m
 \end{equation}
which is $\bar{\mathcal{P}}$-measurable, for each $(y,z,k,y',z',k')$ in $\mathbb{R}^m\times\mathbb{R}^{m\times d}\times L^2(E,\mathcal{B}(E),\lambda;\mathbb{R}^m)\times\mathbb{R}^m\times\mathbb{R}^{m\times d}\times L^2(E,\mathcal{B}(E),\lambda;\mathbb{R}^m)$. Moreover, we also make the following assumptions on $f$:\\
(i) There exists a constant $C\geq0$ such that, $\bar{P}\mbox{-a.s.}$, for all $t\in[0,T]$, $y_1,y_2 ,y_1',y_2'\in\mathbb{R}^m$, $z_1,z_2 ,z_1',z_2'\in\mathbb{R}^{m\times d}$, $k_1,k_2 ,k_1',k_2'\in L^2(E,\mathcal{B}(E),\lambda;\mathbb{R}^m)$,\\
$\indent |f(t,y_1,z_1,k_1,y_1',z_1',k_1')-f(t,y_2,z_2,k_2,y_2',z_2',k_2')|$\\
$\indent\leq C(|y_1-y_2|+|y_1'-y_2'|+|z_1-z_2|+|z_1'-z_2'|+|k_1-k_2|_{L^2(\lambda)}+|k_1'-k_2'|_{L^2(\lambda)}).$\\
(ii) $|f(\cdot,0,0,0,0,0,0)|\in\mathcal{H}_{\bar{\mathbb{F}}}^2(0,T;\mathbb{R}^m)$.\qquad\qquad\indent\indent
\indent\indent\indent\qquad\indent\indent\indent(H2.1)

\begin{lemma}\label{Thm2.1}
Under the assumption {\rm(H2.1)}, for any random variable $\xi\in L^2(\Omega,\mathcal{F}_T,P)$, the mean-field BSDE with jumps
 \begin{equation}
 \begin{aligned}
y(t)=\xi+\int_t^TE'[f(s,y(s),z(s),k(s),y(s)',z(s)',k(s)')]ds-\int_t^Tz(s)dB_s-\int_t^T\int_E k(s,e)\widetilde{\mu}(ds,de),\,\,0\leq t\leq T,
 \end{aligned}
\end{equation}
has a unique adapted solution
$$(y(t),z(t),k(t))_{t\in[0,T]}\in\mathcal{S}_{\mathbb{F}}^2(0,T;\mathbb{R}^m)\times
\mathcal{H}_{\mathbb{F}}^2(0,T;\mathbb{R}^{m\times d})\times\mathcal{K}_{\mathbb{F},\lambda}^2(0,T;\mathbb{R}^m).$$

\end{lemma}
For the proof the readers may refer to ~\cite{LM}.
\begin{remark}\label{rem2.1}
From above notions, the generator of above mean-field BSDE has to be understood as follows
 $$ \begin{array}{rcl}
 &\qquad E'[f(s,y(s),z(s),k(s),y(s)',z(s)',k(s)')](\omega) =E'[f(s,y(s,\omega),z(s,\omega),k(s,\omega),(y(s))',(z(s))')]\\
&=\int_\Omega f(\omega',\omega,s,y(s,\omega),z(s,\omega),y(s,\omega'),z(s,\omega'))P(d\omega'),\,\,\,\,\,\,\omega\in\Omega.
 \end{array}$$
\end{remark}
\begin{remark}\label{rem2.2}
If we assume\\
  {\rm(i)} For each fixed $(x,x',e)\in\mathbb{R}^n\times\mathbb{R}^n\times E$, $b(\cdot,x,x')$, $\sigma(\cdot,x,x')$ and $\gamma(\cdot,x,x',e)$ are continuous in $t$;\\
  {\rm(ii)} There exists a $C> 0$ such that, for all $t\in[0,T]$, $x_1,x_2, x_1',x_2'\in\mathbb{R}^n$,
$$|b(t,x_1,x_1')-b(t,x_2,x_2')|+|\sigma(t,x_1,x_1')-\sigma(t,x_2,x_2')|\leq C(|x_1-x_2|+|x_1'-x_2'|);$$
{\rm{(iii)}} There exists $\rho$: $E\rightarrow\mathbb{R}^+$ with $\int_E\rho^2(e)\lambda(de)<+\infty$, such that, for any $t\in[0,T]$, $x_1,x_2, x_1',x_2'\in\mathbb{R}^n$ and $e\in E$,\\
\indent\qquad$|\gamma(t,x_1,x_1',e)-\gamma(t,x_2,x_2',e)|\leq\rho(e)(|x_1-x_2|+|x_1'-x_2'|),$\\
\indent\qquad$|\gamma(t,0,0,e)|\leq\rho(e)$.

\noindent Then, for any random variable $(t,\zeta)\in[0,T]\times L^2(\Omega, \mathcal{F}_t, P;\mathbb{R}^n)$, the following mean-field SDE with jumps:
\be
x(s)=\zeta+ \int_t^s E'[b(r,x(r),(x(r))')]dr+\int_t^sE'[\sigma(r,x(r),(x(r))')]dB_r
+\int_t^s\int_EE'[\gamma(r,x(r-),(x(r-))',e)]\widetilde{\mu}(dr,de),
 \ee
has a unique adapted solution $x\in S^2_{\mathbb{F}}(0,T; \mathbb{R}^n)$.

For more details, the reader is referred to, e.g., \cite{LM}.
\end{remark}

\section{Mean-field FBSDE with jumps: Existence and uniqueness }
We consider the following fully coupled mean-field forward-backward stochastic differential equations with jumps:
\begin{equation}\label{Eq3.1}
\left\{
\begin{aligned}
dx(t)=&\ \int_EE'[b(t,x(t),y(t),z(t),k(t,e),(x(t))',(y(t))',(z(t))',(k(t,e))')]\lambda(de)dt\\
&\ +\int_EE'[\sigma(t,x(t),y(t),z(t),k(t,e),(x(t))',(y(t))',(z(t))',(k(t,e))')]\lambda(de)dB_t\\
&\ +\int_EE'[h(t,x(t-),y(t-),z(t),k(t,e),(x(t-))',(y(t-))',(z(t))',(k(t,e))',e)]\widetilde{\mu}(dt,de),\\
-dy(t)=&\ \int_EE'[f(t,x(t),y(t),z(t),k(t,e),(x(t))',(y(t))',(z(t))',(k(t,e))')]\lambda(de)dt\\
&\ -z(t)dB_t-\int_Ek(t,e)\widetilde{\mu}(dt,de),\, t\in[0,T],\\
x(0)=&\ a,\\
y(T)=&\ E'[\Phi(x(T),(x(T))')],
\end{aligned}
\right.
\end{equation}
%where the solution $(X,Y,Z,K)$ takes its values in $\mathbb{R}^n\times\mathbb{R}^m\times\mathbb{R}^{m\times d}\times\mathbb{R}^m$,
where the coefficients:
$$\begin{array}{lll}
&&b:\bar{\Omega} \times[0,T] \times \mathbb{R}^n \times \mathbb{R}^m \times \mathbb{R}^{m\times d}\times\mathbb{R}^m\times \mathbb{R}^n \times \mathbb{R}^m \times \mathbb{R}^{m\times d}\times\mathbb{R}^m\rightarrow \mathbb{R}^n,\\
&&\sigma:\bar{\Omega} \times[0,T]\times \mathbb{R}^n \times \mathbb{R}^m \times \mathbb{R}^{m\times d}\times\mathbb{R}^m\times \mathbb{R}^n \times \mathbb{R}^m \times \mathbb{R}^{m\times d}\times\mathbb{R}^m\rightarrow \mathbb{R}^{n\times d},\\
&&h:\bar{\Omega} \times [0,T]\times \mathbb{R}^n \times \mathbb{R}^m \times \mathbb{R}^{m\times d}\times\mathbb{R}^m\times \mathbb{R}^n \times \mathbb{R}^m \times \mathbb{R}^{m\times d}\times\mathbb{R}^m\times E\rightarrow \mathbb{R}^{n},\\
&&f:\bar{\Omega} \times [0,T]\times \mathbb{R}^n \times \mathbb{R}^m \times \mathbb{R}^{m\times d}\times\mathbb{R}^m\times \mathbb{R}^n \times \mathbb{R}^m \times \mathbb{R}^{m\times d}\times\mathbb{R}^m\rightarrow \mathbb{R}^{m},\\
&&\Phi: \bar{\Omega} \times \mathbb{R}^n\times\mathbb{R}^n\rightarrow \mathbb{R}^m.
\end{array}
$$
%\begin{remark}\label{rem3.0}
%In Li~\cite{L}, the author studied the stochastic maximum principle in mean-field controls, the related feedback control system takes a special case of the mean-field FBSDE (\ref{Eq3.1}).
%\end{remark}
Given an $m\times n$ full-rank matrix $G$. We use the following notations
$$
\mathbf{\lambda}=\left( {\begin{array}{c}
   x\\
   y\\
   z\\
   k\\
\end{array}} \right), \ \  \mathbf{\widetilde{\lambda}}=\left( {\begin{array}{c}
   \widetilde{x}\\
   \widetilde{y}\\
   \widetilde{z}\\
   \widetilde{k}\\
\end{array}} \right), \ \
\mathbf{A(t,\lambda,\widetilde{\lambda},e)}=\left( {\begin{array}{c}
   -G^T f(t,\lambda,\widetilde{\lambda})\\
   Gb(t,\lambda,\widetilde{\lambda})\\
   G\sigma(t,\lambda,\widetilde{\lambda}) \\
   Gh(t,\lambda,\widetilde{\lambda},e)\\
\end{array}} \right)
%\left(t,\lambda,\widetilde{\lambda}\right),
$$
where $G\sigma=(G\sigma_1,\cdots,G\sigma_d)$. We use the standard inner product and Euclidean norm in $\mathbb{R}^{m\times d}$.
\begin{definition}\label{def 3.1}
A quadruple of processes $(X,Y,Z,K)$ is called an adapted solution of mean-field FBSDE with jumps (\ref{Eq3.1}), if $(X,Y,Z,K)\in \mathcal{H}^2_{\mathbb{F}}(0,T;\mathbb{R}^n\times\mathbb{R}^m\times\mathbb{R}^{m\times d})\times\mathcal{K}^2_{\mathbb{F},\lambda}(0,T;\mathbb{R}^m)$ and satisfies equation (\ref{Eq3.1}).
\end{definition}
We assume that
\begin{eqnarray*}{\rm(H3.1)}
 \begin{array}{llll}
&{\rm{(i)}}\  A(t,\lambda,\widetilde{\lambda},e) \mbox{ is uniformly Lipschitz with respect to } \lambda, \widetilde{\lambda};\\
& {\rm{(ii)}}\ \mbox{The coefficients } (b,\sigma,h,f) \mbox{ are uniformly Lipschitz in } (x,y,z,k,\widetilde{x},\widetilde{y},\widetilde{z},\widetilde{k});\\
& {\rm{(iii)}}\ \mbox{for each }\lambda,\,\widetilde{\lambda}, \ A(\cdot,\lambda,\widetilde{\lambda}) \mbox{ is in}\ \mathcal{M}_{\bar{\mathbb{F}}}^2(0,T);\\
& {\rm{(iv)}}\ \Phi(x,\widetilde{x}) \mbox{ is uniformly Lipchitz with respect to } x, \widetilde{x} \in \mathbb{R}^n;\\
& {\rm{(v)}}\ \mbox{for each }(x,y,z,k,\widetilde{x},\widetilde{y},\widetilde{z},\widetilde{k})\in \mathbb{R}^n \times \mathbb{R}^m \times \mathbb{R}^{m\times d}\times\mathbb{R}^m\times \mathbb{R}^n \times \mathbb{R}^m \times \mathbb{R}^{m\times d}\times\mathbb{R}^m, \\
 & \indent \Phi(x,\widetilde{x}) \in L^2(\bar{\Omega},\mathcal{\bar{F}}_T,\bar{P};\mathbb{R}^m);\\
   & \indent b,\sigma,h,f \mbox{ are } \bar{\mathbb{F}}\mbox{-progressively measurable}; \\
  & \indent h(\cdot,0,0,0,0,0,0,0,0,\cdot)\in \mathcal{K}^2_{\bar{\mathbb{F}},\lambda}(0,T).
  \end{array}
\end{eqnarray*}

We also need the following monotonicity assumptions. For any $ \lambda=(x,y,z,k)^T,\  \widetilde{\lambda}=(\widetilde{x},\widetilde{y},\widetilde{z},\widetilde{k})^T, \  \bar{\lambda}=(\bar{x},\bar{y},\bar{z},\bar{k})^T,$ $\widehat{l}=l-\bar{l}, \mbox{where } l=x, y, z,k, \widetilde{x}, \widetilde{y}, \widetilde{z}, \widetilde{k}, \mbox{ respectively,}$ it holds, $\bar{P}$-a.s.,
\begin{eqnarray*}{\rm(H3.2)}
 \begin{array}{ll}
\rm{(i)}& \int_E\langle A(t,\lambda,\widetilde{\lambda},e)-A(t,\bar{\lambda},\widetilde{\lambda},e),\lambda-\bar{\lambda}\rangle\lambda(de)\leq -\beta_1|\widehat{x}|^2-\beta_2(|\widehat{y}|^2+|\widehat{z}|^2)-\beta_3\int_E|\widehat{k}(e)|^2\lambda(de),\\
 \rm{(ii)}&\langle\Phi(x,\widetilde{x})-\Phi(\bar{x},\widetilde{x}),G(x-\bar{x})\rangle\geq \mu_1|\widehat{x}|^2,\\
  \end{array}
\end{eqnarray*}
where $\beta_1$, $\beta_2$, $\beta_3$ and $\mu_1$ are given nonnegative constants with
$$(1)\ \beta_1-L_AC_0>0,\ \beta_2-L_AC_0\geq0,\ \beta_3-L_A\geq 0, \mu_1-L_\Phi\lambda_1>0,$$
or
$$(2)\ \beta_1-L_AC_0=0,\ \beta_2-L_AC_0>0,\ \beta_3-L_A> 0,\ \mu_1-L_\Phi\lambda_1>0,$$
where $L_A,\,L_\Phi$ are the Lipschitz constants of $A,\,\Phi$ with respect to $\widetilde{\lambda}$, $\widetilde{x}$, respectively; $C_0$ and $\lambda_1$ satisfy $\int_E 1 \lambda(de)\leq C_0$ and $|G\widehat{l}(T)|\leq\lambda_1|\widehat{l}(T)|$, respectively.

Then we have the following  main result in this section.

\begin{theorem}\label{Thm3.1}
We assume (H3.1) and (H3.2) hold, then mean-field FBSDE with jumps (\ref{Eq3.1}) has a unique adapted solution $(X,Y,Z,K)$.
\end{theorem}

%\begin{remark}\label{Thm3.2}
% Similarly, if (H3.1) and (H3.3) hold, then mean-field FBSDE with jumps (\ref{Eq3.1}) has a unique adapted solution $(X,Y,Z,K)$.
%\end{remark}

\begin{proof}
We first prove the uniqueness of the solution. Let $\lambda(t,e)=(x(t),y(t),z(t),k(t,e))$ and $\bar{\lambda}(t,e)=(\bar{x}(t),\bar{y}(t),\bar{z}(t),\bar{k}(t,e))$ be two solutions of equation (\ref{Eq3.1}). We set $\widehat{l}=l-\bar{l}$, where $l=x(t), y(t), z(t),k(t,e),\widetilde{x}(t),$ $\widetilde{y}(t),\widetilde{z}(t),\widetilde{k}(t,e)$, respectively. Applying It\^{o}'s formula to $\langle G\widehat{x}(s),\widehat{y}(s) \rangle$, we get
 \begin{equation}\nonumber
 \begin{aligned}
&\ E \Big\langle E'[\Phi(x(T),(x(T))')]-E'[\Phi(\bar{x}(T),(\bar{x}(T))')] ,  G(x(T)-\bar{x}(T)) \Big\rangle \\
=&\ E \int_0^T \int_E\Big\langle E'[A(t,\lambda(t,e),(\lambda(t,e))',e)]-E'[A(t,\bar{\lambda}(t,e),(\bar{\lambda}(t,e))',e)], \lambda(t,e)-\bar{\lambda}(t,e)\Big\rangle  \lambda(de)dt.
\end{aligned}
\end{equation}
From (H3.2) the monotonicity assumptions of $\Phi$ and $A$, we get
\begin{equation}\label{equ 2018102201}
\begin{aligned}
(\mu_1-L_\Phi\lambda_1)E[|\widehat{x}(T)|^2] \leq &\ -E\int_0^T\[\beta_1|\widehat{x}(t)|^2+\beta_2(|\widehat{y}(t)|^2+|\widehat{z}(t)|^2)
+\beta_3\int_E|\widehat{k}(t,e)|^2\lambda(de)\]dt\\
&\ +L_AC_0E\int_0^T\[|\widehat{x}(t)|^2+|\widehat{y}(t)|^2+|\widehat{z}(t)|^2\]dt
+L_AE\int_0^T\int_E|\widehat{k}(t,e)|^2\lambda(de)dt.
\end{aligned}
\end{equation}
(1) When $\beta_1-L_AC_0>0,\ \beta_2-L_AC_0\geq0,\ \beta_3-L_A\geq 0, \ \mu_1-L_\Phi\lambda_1>0$, from (\ref{equ 2018102201}) we can get
%$|\widehat{x}(t)|^2=0$, dsdP-a.e. In this case we have $\widehat{x}(t)=0$, P-a.s., for all $t\in [0, T]$.
$$|\widehat{x}(t)|^2=0, dtdP\text{-a.e.},\ |\widehat{x}(T)|^2=0, P\text{-a.s.}$$
Thus, $\Phi(x(T),(x(T))')=\Phi(\bar{x}(T),(\bar{x}(T))')$, $\bar{P}$-a.s. Therefore, from Lemma \ref{Thm2.1} it follows that
$$ ||\widehat{y}||_{\mathcal{S}_{\mathbb{F}}^2}=0,\ ||\widehat{z}||_{\mathcal{H}_{\mathbb{F}}^2}=0,\ ||\widehat{k}||_{\mathcal{K}_{\mathbb{F},\lambda}^2}=0.$$
(2) When $\beta_1-L_AC_0=0,\ \beta_2-L_AC_0>0,\ \beta_3-L_A>0,\ \mu_1-L_\Phi\lambda_1>0$, from (\ref{equ 2018102201}) we can get
$$||\widehat{y}||_{\mathcal{S}_{\mathbb{F}}^2}=0,\ ||\widehat{z}||_{\mathcal{H}_{\mathbb{F}}^2}=0,\ ||\widehat{k}||_{\mathcal{K}_{\mathbb{F},\lambda}^2}=0,\ x(T)=\bar{x}(T),\ P\text{-a.s.}$$
 From the uniqueness of solutions of mean-field SDEs with jumps (refer to \cite{LM}, or Remark \ref{rem2.2}), we get $x(t)=\bar{x}(t)$, P-a.s., for all $t\in [0, T]$.

We now prove the existence of the solution. For this we introduce
  the following mean-field FBSDEs with jumps parameterized by $\alpha\in[0,1]$:
\begin{equation}\label{Eq3.2}
\left\{
\begin{aligned}
dx^{\alpha}(t)=&\ \[\alpha \int_EE'[b(t,\chi^{\alpha}(t,e))]\lambda(de)+E'[\phi(t)]\]dt +\[\alpha \int_EE'[\sigma(t,\chi^{\alpha}(t,e))]\lambda(de)+E'[\psi(t)]\]dB_t\\
&\ +\int_E \[\alpha E'[h(t,\chi^{\alpha}(t-,e))]+E'[\varphi(t,e)]\]\widetilde{\mu}(dt,de),\\
-dy^{\alpha}(t)=&\ \[(1-\alpha)\beta_1Gx^{\alpha}(t)+\alpha\int_E E'[f(t,\chi^{\alpha}(t,e))]\lambda(de) +E'[\gamma(t)]\]dt -z^{\alpha}(t)dB_t-\int_E k^{\alpha}(t,e)\widetilde{\mu}(dt,de),\\
x^{\alpha}(0)=&\ a,\\
y^{\alpha}(T)=&\alpha E'[\Phi(x^{\alpha}(T),(x^{\alpha}(T))')]+(1-\alpha)Gx^{\alpha}(T)+\xi,
\end{aligned}
\right.
\end{equation}
where
$\chi^{\alpha}(t,e)=(x^{\alpha}(t),y^{\alpha}(t),z^{\alpha}(t),k^{\alpha}(t,e),(x^{\alpha}(t))',(y^{\alpha}(t))',(z^{\alpha}(t))',
(k^{\alpha}(t,e))')$, $\chi^{\alpha}(t-,e)=(x^{\alpha}(t-),y^{\alpha}(t-),\\ z^{\alpha}(t),k^{\alpha}(t,e),(x^{\alpha}(t-))',(y^{\alpha}(t-))',(z^{\alpha}(t))',
(k^{\alpha}(t,e))',e)$; $\phi,\ \psi$ and $\gamma$ are given processes in $\mathcal{H}_{\bar{\mathbb{F}}}^2(0,T)$ with values in $\mathbb{R}^n, \ \mathbb{R}^{n\times d}$ and $\mathbb{R}^m $, respectively; $\varphi\in\mathcal{K}^2_{\mathbb{F},\lambda}(0,T;\mathbb{R}^n)$ and $\xi\in L^2(\Omega, \mathcal{F}_T, P)$.

When $\alpha=0$,  from the existence and the uniqueness of the solutions of Mckean-Vlasov equation with jumps and mean-field BSDE with jumps we know equation (\ref{Eq3.2}) has a unique solution. Then from Lemma \ref{Lem3.1} in Appendix, there
exists a positive constant $\delta_0$ depending on Lipschitz constants, $\beta_1,\ \beta_2,\ \beta_3,\ \mu_1,\ \lambda_1$ and $T$, such that, for every $\delta\in[0,\delta_0]$, equation  (\ref{Eq3.2}) for $\alpha=\delta$ has a unique solution. We can repeat this process $N$ times where $1\leq N\delta_0\leq 1+\delta_0$. It means that, in particular, mean-field FBSDE (\ref{Eq3.2}) for $\alpha=1$ has a unique solution, i.e., (\ref{Eq3.1}) has a unique solution. \\
 The proof is complete.
\end{proof}
\begin{remark}
We note that the existence and the uniqueness of the solution of our equation (\ref{Eq3.1}) can also be obtained if  the monotonicity assumption (H3.2) in Theorem \ref{Thm3.1} is changed by the following form
\begin{eqnarray*}{\rm(H3.3)}
\begin{array}{ll}
\rm{(i)}& \int_E\langle A(t,\lambda,\widetilde{\lambda},e)-A(t,\bar{\lambda},\widetilde{\lambda},e),\lambda-\bar{\lambda}\rangle\lambda(de) \geq \beta_1|\widehat{x}|^2+\beta_2(|\widehat{y}|^2+|\widehat{z}|^2)+\beta_3\int_E|\widehat{k}(e)|^2\lambda(de);\\
 \rm{(ii)}&\langle\Phi(x,\widetilde{x})-\Phi(\bar{x},\widetilde{x}),G(x-\bar{x})\rangle \leq -\mu_1|\widehat{x}|^2;\\
  \end{array}
\end{eqnarray*}
where $\beta_1$, $\beta_2$, $\beta_3$ and $\mu_1$ are given nonnegative constants with $\beta_1-L_AC_0>0,\ \beta_2-L_AC_0\geq0$,\ $\beta_3-L_A\geq 0$, $\mu_1-L_\Phi\lambda_1>0$, or $\beta_1-L_AC_0=0,\ \beta_2-L_AC_0>0$,\ $\beta_3-L_A> 0$, $\mu_1-L_\Phi\lambda_1>0$, where $L_A,\,L_\Phi$,\  $C_0$ and $\lambda_1$ are the same as those in (H3.2).

%\begin{remark}\label{rem3.1}
The proof of this result is similar to that of Theorem \ref{Thm3.1} but one need to notice that the equation (\ref{Eq3.2}) should be changed into the following form
\begin{equation}\nonumber
\left\{
\begin{aligned}
dx^{\alpha}(t)=&\ \[\alpha\int_E E'[b(t,\chi^{\alpha}(t,e))]\lambda(de)+E'[\phi(t)]\]dt+\[\alpha \int_E E'[\sigma(t,\chi^{\alpha}(t,e))]\lambda(de)+E'[\psi(t)]\]dB_t\\
&\ +\int_E\[\alpha E'[h(t,\chi^{\alpha}(t,e),e)]+E'[\varphi(t,e)]\]\widetilde{\mu}(dt,de),\\
-dy^{\alpha}(t)=&\ \[-(1-\alpha)\beta_1Gx^\alpha(t)+\alpha\int_E E'[f(t,\chi^{\alpha}(t,e))]\lambda(de) +E'[\gamma(t)]\]dt-z^{\alpha}(t)dB_t-\int_E k^{\alpha}(t,e)\widetilde{\mu}(dt,de),\\
x^{\alpha}(0)=&\ a,\\
y^{\alpha}(T)=& \alpha E'[\Phi(x^{\alpha}(T),(x^{\alpha}(T))')]-(1-\alpha)Gx^{\alpha}(T)+\xi.
\end{aligned}
\right.
\end{equation}
%\end{remark}

\end{remark}

\begin{remark}\label{rem3.2}
(i)\ When $\Phi$ does not depends on $x$, $\widetilde{x}$, i.e., $\Phi(x,\widetilde{x})=\xi\in L^2(\Omega, \mathcal{F}_T, P)$ is given, for the existence and the uniqueness of the solution of mean-field FBSDE (\ref{Eq3.1}), the monotonicity assumption (H3.2) can be weakened as
\begin{equation}\nonumber
\int_E\langle A(t,\lambda,\widetilde{\lambda},e)-A(t,\bar{\lambda},\widetilde{\lambda},e),\lambda-\bar{\lambda}\rangle\lambda(de) \leq-\beta_1|\widehat{x}|^2-\beta_2|\widehat{y}|^2;
\end{equation}
%similarly, (H3.3) can be weakened as
%\begin{equation}\nonumber
%\int_E\langle A(t,\lambda,\widetilde{\lambda})-A(t,\bar{\lambda},\widetilde{\lambda}),\lambda-\bar{\lambda}\rangle \lambda(de) \geq\beta_1|\widehat{x}|^2+\beta_2|\widehat{y}|^2;
%\end{equation}
 where $\beta_1$ and $\beta_2$  are given nonnegative constants with $\beta_1-C_0L_A\geq0,\ \beta_2-C_0L_A\geq0$ (the equalities can not be established at the same time),  $L_A$ is the Lipchitz constants of $A$ with respect to $\widetilde{\lambda}$.

(ii)\ When $\sigma$ does not depends on $z,\,z',\, k,\, k'$, the mean-field FBSDE (3.1) also has a unique adapted solution, but the monotonicity (H3.2) should be weakened as
\begin{eqnarray*}
\begin{array}{ll}
\rm{(i)}& \int_E\langle A(t,\lambda,\widetilde{\lambda},e)-A(t,\bar{\lambda},\widetilde{\lambda},e),\lambda-\bar{\lambda}\rangle\lambda(de) \leq -\beta_1|\widehat{x}|^2;\\
 \rm{(ii)}&\langle\Phi(x,\widetilde{x})-\Phi(\bar{x},\widetilde{x}),G(x-\bar{x})\rangle\geq \mu_1|\widehat{x}|^2,
  \end{array}
\end{eqnarray*}
%similarly, (H3.3) can be weakened as
%\begin{eqnarray*}
%\begin{array}{ll}
%\rm{(i)}& \int_E\langle A(t,\lambda,\widetilde{\lambda})-A(t,\bar{\lambda},\widetilde{\lambda}),\lambda-\bar{\lambda}\rangle\lambda(de) \geq \beta_1|\widehat{x}|^2;\\
% \rm{(ii)}&\langle\Phi(x,\widetilde{x})-\Phi(\bar{x},\widetilde{x}),G(x-\bar{x})\rangle \leq -\mu_1|\widehat{x}|^2;\\
%  \end{array}
%\end{eqnarray*}
where $\beta_1$ and $\mu$ are given nonnegative constants with $\beta_1>L_A+2L_AC_{L_g,T}C_0^2$, $\mu_1>L_\Phi\lambda_1+8C_{L_g,T}L_\Phi^2L_AC_0$ ( $C_{L_g,T}:=\exp{\{[C_0(4L_f+12L_f^2+8L_f^2C_0)+1]T}\}$).
\end{remark}

\begin{example}
We consider
\begin{equation}\nonumber
\left\{
\begin{aligned}
dx(t)=&\  E'[-y'(t)-2y(t)]dt+E'[-z'(t)-2z(t)]dB_t+\int_EE'[-k'(t,e)-2k(t,e)]\widetilde{\mu}(dt,de),\,t\in[0,T],\\
-dy(t)=&\ E'[x'(t)+2x(t)]dt-z(t)dB_t-\int_Ek(t,e)\widetilde{\mu}(dt,de),\,t\in[0,T],\\
x(0)=&\ 1,\\
y(T)=& E'[x'(T)+2x(T)].
\end{aligned}
\right.
\end{equation}
We can take $\beta_1=\beta_2=\beta_3=2$, $\mu_1=2$, $C_0=1$, $L_A=1$, $L_\Phi=1$, from Theorem {\rm{3.1}}, we know it has a unique solution.
\end{example}
We now give an example to explain that the assumption (H3.2)  is necessary for Theorem \ref{Thm3.1}, i.e., if the coefficients of our equation do not satisfy (H3.2),
the solution of equation (\ref{Eq3.1}) may not exist.
\begin{example}
We take $m=n=d=1$ here. We consider
\begin{equation}\label{Eq E.1}
\left\{
\begin{aligned}
dx(t)=&\ E[y(t)]dt +dB_t+\int_Ek(t,e)\widetilde{\mu}(dt,de),\ t\in [0,\frac{3}{4}\pi],\\
-dy(t)=&\ E[x(t)]dt-z(t)dB_t-\int_Ek(t,e)\widetilde{\mu}(dt,de),\ t\in [0,\frac{3}{4}\pi],\\
x(0)=&\ 1,\ y(\frac{3}{4}\pi)=-E[x(\frac{3}{4}\pi)],\ t\in [0, \frac{3}{4}\pi].
\end{aligned}
\right.
\end{equation}
It's easy to check this equation does not satisfy (H3.2), we point out that it also does not exist an adapted solution. In fact, if $(x,y,z,k)_{0\leq t\leq \frac{3}{4}\pi}$ is the solution of mean-field FBSDE (\ref{Eq E.1}), then $(E[x(t)], E[y(t)])$ is the solution of the following ordinary differential equation (ODE, for short):\\
\begin{equation}\label{Eq E.2}
\left\{
\begin{aligned}
\dot{X}=&\ Y,\  \dot{Y}= -X,\\
 X(0)=&\ 1,\ Y(\frac{3}{4}\pi)=-X(\frac{3}{4}\pi), t\in [0, \frac{3}{4}\pi].
\end{aligned}
\right.
\end{equation}
But we know this ODE has no solution, therefore there is no adapted solution of (\ref{Eq E.1}).
\end{example}

\section{Continuity property on the parameters}
 \qquad $\,\,\,\,$In this section we will discuss the continuity of the solution of equation (\ref{Eq3.1}) depending on parameters. We consider the following mean-field FBSDEs with coefficients $(b_{\alpha}, \sigma_{\alpha},h_{\alpha}, f_{\alpha}, \Phi_{\alpha}),\ \alpha\in \mathbb{R}$:
\begin{equation}\label{Eq4.1}
\left\{
\begin{aligned}
dx^{\alpha}(t)=&\ \int_EE'[b_\alpha(t,\chi^{\alpha}(t,e))]\lambda(de)dt+\int_EE'[\sigma_\alpha(t,\chi^{\alpha}(t,e))]\lambda(de)dB_t
+\int_EE'[h_\alpha(t,\chi^{\alpha}(t-,e))]\widetilde{\mu}(dt,de),\\
-dy^{\alpha}(t)=&\ \int_EE'[f_\alpha(t,\chi^{\alpha}(t,e))]\lambda(de)dt-z^{\alpha}(t)dB_t-\int_Ek^{\alpha}(t,e)\widetilde{\mu}(dt,de),\\
x^{\alpha}(0)=&\ a,\\
y^{\alpha}(T)=&\ E'[\Phi_\alpha(x^{\alpha}(T),(x^{\alpha}(T))')],
\end{aligned}
\right.
\end{equation}
where $$\chi^{\alpha}(t,e)=(x^{\alpha}(t),y^{\alpha}(t),z^{\alpha}(t),k^{\alpha}(t,e),(x^{\alpha}(t))',
(y^{\alpha}(t))',(z^{\alpha}(t))',(k^{\alpha}(t,e))'),$$
$$\chi^{\alpha}(t-,e)=(x^{\alpha}(t-),y^{\alpha}(t-), z^{\alpha}(t),k^{\alpha}(t,e),(x^{\alpha}(t-))',
(y^{\alpha}(t-))',(z^{\alpha}(t))',(k^{\alpha}(t,e))',e),$$
and the mappings $b_{\alpha}, \sigma_{\alpha},h_{\alpha}, f_{\alpha}, \Phi_{\alpha}, A_\alpha=(-G^T f_\alpha, Gb_\alpha, G\sigma_\alpha, G h_\alpha)^T$
%$\alpha\in \mathbb{R},$
 satisfy (H3.1) and (H3.2), for each $\alpha \in \mathbb{R}$. Then, from Theorem 3.1 we know mean-field FBSDE (\ref{Eq4.1}) has a unique solution $(x^{\alpha}, y^{\alpha}, z^{\alpha},k^{\alpha})$\ for each $\alpha\in \mathbb{R}.$\\
Let us give some more assumptions.
\begin{eqnarray*}{\rm(H4.1)}
\begin{array}{lll}
\rm{(i)}& \mbox{The coefficients } (b_{\alpha}, \sigma_{\alpha},h_{\alpha}, f_{\alpha}, \Phi_{\alpha}), \alpha\in \mathbb{R},\ \mbox{are
 uniformly Lipschitz in } (x, y, z,k, \widetilde{x}, \widetilde{y}, \widetilde{z},\widetilde{k});\\
 \rm{(ii)}& \mbox{The mappings}\ \alpha\mapsto (b_{\alpha}, \sigma_{\alpha},h_{\alpha}, f_{\alpha}, \Phi_{\alpha}),\ \alpha\in \mathbb{R},\ \mbox{are continuous respectively}.
\end{array}
\end{eqnarray*}
Then we have the following continuity property.

\begin{theorem}\label{Thm4.1}
Let the coefficients $(b_{\alpha}, \sigma_{\alpha}, h_{\alpha},f_{\alpha}, \Phi_{\alpha}), \alpha\in \mathbb{R}$, satisfy (H3.1), (H3.2) and (H4.1), and the associated solution of mean-field FBSDE with jumps (\ref{Eq4.1}) is denoted by $(x^\alpha, y^\alpha, z^\alpha,k^\alpha)$. Then, the mappings
 $$\alpha\mapsto(x^\alpha, y^\alpha, z^\alpha,k^\alpha, x^{\alpha}(T)): \mathbb{R}\mapsto \mathcal{H}_{\mathbb{F}}^2(0,T;\mathbb{R}^n\times\mathbb{R}^m\times\mathbb{R}^{m\times d})\times\mathcal{K}^2_{\mathbb{F},\lambda}(0,T;\mathbb{R}^m)
 \times L^2(\Omega,\mathcal{F}_T,P;\mathbb{R}^n)$$
 is continuous.
\end{theorem}

\begin{proof} For simplicity of notations, we only prove the continuity of the solutions $(x^\alpha, y^\alpha, z^\alpha,k^\alpha, x^{\alpha}(T))$\ of mean-field FBSDE (\ref{Eq4.1}) at $\alpha=0$. We want to prove that
$(x^\alpha, y^\alpha, z^\alpha,k^\alpha, x^{\alpha}(T))$ converges to
$(x^0, y^0, z^0,k^0, x^0(T))$ in $\mathcal{H}_{\mathbb{F}}^2(0,T;\mathbb{R}^n\times\mathbb{R}^m\times\mathbb{R}^{m\times d})\times\mathcal{K}^2_{\mathbb{F},\lambda}(0,T;\mathbb{R}^m)
\times L^2(\Omega,\mathcal{F}_T,P;\mathbb{R}^n)$
as $\alpha$ tends to 0. We set $\lambda^\alpha(t,e)=(x^\alpha(t), y^\alpha(t), z^\alpha(t),k^\alpha(t,e))$, and  $\widehat{\lambda}(t,e)=\lambda^\alpha(t,e)-\lambda^0(t,e)=(\widehat{x}(t),\widehat{y}(t),\widehat{z}(t),\widehat{k}(t,e))
=(x^\alpha(t)-x^0(t), y^\alpha(t)-y^0(t), z^\alpha(t)-z^0(t),k^\alpha(t,e))-k^0(t,e))$, then from (\ref{Eq4.1}) we know
\begin{equation}
\left\{
\begin{aligned}
d\widehat{x}(t)=&\ \int_EE'\[b_\alpha(t,\lambda^{\alpha}(t,e),(\lambda^{\alpha}(t,e))')-b_0(t,\lambda^0(t,e),(\lambda^0(t,e))')\]\lambda(de)dt\\
&+\ \int_EE'\[\sigma_\alpha(t,\lambda^{\alpha}(t,e),(\lambda^{\alpha}(t,e))')-\sigma_0(t,\lambda^0(t,e),(\lambda^0(t,e))')\]
\lambda(de)dB_t\\
&+\ \int_EE'\[h_\alpha(t,\lambda^{\alpha}(t,e),(\lambda^{\alpha}(t,e))',e)-h_0(t,\lambda^0(t,e),(\lambda^0(t,e))',e)\]\widetilde{\mu}(dt,de),\\
-d\widehat{y}(t)=&\ \int_EE'\[f_\alpha(t,\lambda^{\alpha}(t,e),(\lambda^{\alpha}(t,e))')-f_0(t,\lambda^0(t,e),(\lambda^0(t,e))')\]\lambda(de)dt\\
&-\widehat{z}(t)dB_t-\int_E\widehat{k}(t,e)\widetilde{\mu}(dt,de),\\
\widehat{x}(0)=&\ 0,\\
\widehat{y}(T)=&E'\[\Phi_\alpha(x^{\alpha}(T),(x^{\alpha}(T))')-\Phi_0(x^0(T),(x^0(T))')\].
\end{aligned}
\right.
\end{equation}
From assumptions (H3.1), (H3.2) and (H4.1), and standard estimates of $\widehat{x}(t)$ and
$(\widehat{y}(t), \widehat{z}(t),\widehat{k}(t))$, we get
  \begin{equation}\label{equ 2018103103}
\sup_{0\leq t\leq T}E|\widehat{x}(t)|^2
\leq C_1 E\int_0^T\(|\widehat{y}(t)|^2+|\widehat{z}(t)|^2+\int_E|\widehat{k}(t,e)|^2\lambda(de)\)dt
+C_1\bar{E}\int_0^T\int_E\[|\widehat{b}(t,e)|^2+|\widehat{\sigma}(t,e)|^2\]\lambda(de)dt;
\end{equation}
\begin{equation}\label{equ 2018103102}
\begin{aligned}
&\ E\int_0^T\(|\widehat{y}(t)|^2+|\widehat{z}(t)|^2+\int_E|\widehat{k}(t,e)|^2\lambda(de)\)dt\\
\leq&\ C_1\Big\{E\int_0^T|\widehat{x}(t)|^2dt+E|\widehat{x}(T)|^2+\bar{E}\int_0^T\int_E|\widehat{f}(t,e)|^2\lambda(de)dt
+\bar{E}[|\widehat{\Phi}(T)|^2]\Big\},
\end{aligned}
\end{equation}
here $C_1$ depends on the Lipchitz constants of $(b_\alpha,\sigma_\alpha,h_\alpha,f_\alpha)$, constant $C_0$ and $T$, where
\begin{equation}\nonumber
\begin{aligned}
\widehat{b}(t,e)=&\ b_{\alpha}(t,\lambda^0(t,e),(\lambda^0(t,e))')-b_0(t,\lambda^0(t,e),(\lambda^0(t,e))'),\\
\widehat{\sigma}(t,e)=&\ \sigma_{\alpha}(t,\lambda^0(t,e),(\lambda^0(t,e))')-\sigma_0(t,\lambda^0(t,e),(\lambda^0(t,e))'),\\
\widehat{h}(t,e)=&\ h_{\alpha}(t,\lambda^0(t,e),(\lambda^0(t,e))',e)-h_0(t,\lambda^0(t,e),(\lambda^0(t,e))',e),\\
\widehat{f}(t,e)=&\ -f_{\alpha}(t,\lambda^0(t,e),(\lambda^0(t,e))')+f_0(t,\lambda^0(t,e),(\lambda^0(t,e))'),\\
\widehat{\Phi}(T)=&\ \Phi_\alpha(x^0(T),(x^0(T))')-\Phi_0(x^0(T),(x^0(T))').
\end{aligned}
\end{equation}
Applying It\^o's formula to $\langle G\widehat{x}(t),\widehat{y}(t)\rangle$ it yields
\begin{equation}\nonumber
\begin{aligned}
&\ E\big\langle E'[\Phi_\alpha(x^\alpha(T),(x^\alpha(T))')-\Phi_\alpha(x^0(T),(x^0(T))')],G\widehat{x}(T)\big\rangle\\
&\ +E\big\langle E'[\Phi_\alpha(x^0(T),(x^0(T))')-\Phi_0(x^0(T),(x^0(T))')],G\widehat{x}(T)\big\rangle\\
=&\ E\int_0^T\int_EE'\big\langle A_\alpha(t,\lambda^\alpha(t,e),(\lambda^\alpha(t,e))',e)-A_\alpha(t,\lambda^0(t,e),(\lambda^0(t,e))',e),\widehat{\lambda}(t,e)
\big\rangle \lambda(de)dt\\
&\ +E\int_0^T\int_EE'\[\langle G\widehat{x}(t),\widehat{f}(t,e)\rangle+\langle G^T\widehat{y}(t),\widehat{b}(t,e)\rangle+\langle G^T\widehat{z}(t),\widehat{\sigma}(t,e)\rangle
+\langle G^T\widehat{k}(t,e),\widehat{h}(t,e)\rangle\]\lambda(de)dt.
\end{aligned}
\end{equation}
With the help of (H3.2) and the Lipschitz properties of $A_\alpha$ and $\Phi_\alpha$, we have
\begin{equation}
 \begin{aligned}
&\ (\mu_1-L_{\Phi_\alpha}\lambda_1)E|\widehat{x}(T)|^2+(\beta_1-C_0L_{A_\alpha})E\int_0^T|\widehat{x}(t)|^2dt+
(\beta_2-C_0L_{A_\alpha})E\int_0^T(|\widehat{y}(t)|^2+|\widehat{z}(t)|^2)dt\\
&\ +(\beta_3-L_{A_\alpha})E\int_0^T\int_E|\widehat{k}(t,e)|^2
\lambda(de)dt\\
\leq &\ C_2E\[E'|\widehat{\Phi}(T)|^2+\int_0^T\int_EE'\(|\widehat{b}(t,e)|^2+|\widehat{f}(t,e)|^2
+|\widehat{\sigma}(t,e)|^2\)
\lambda(de)dt\]\\
&\ +\delta\[E|\widehat{x}(T)|^2+E\int_0^T\(|\widehat{x}(t)|^2+|\widehat{y}(t)|^2+|\widehat{z}(t)|^2
+\int_E|\widehat{k}(t,e)|^2\lambda(de)\)dt\],
\end{aligned}
\end{equation}
for any $\delta>0$.
Since $\beta_1-C_0L_{A_\alpha}>0$, $\beta_2-C_0L_{A_\alpha}\geq0$, $\beta_3-L_{A_\alpha}\geq0$, $\mu_1-L_{\Phi_\alpha}\lambda_1>0$ \big(the situation of  $\beta_1-C_0L_{A_\alpha}=0$, $\beta_2-C_0L_{A_\alpha}>0$, $\beta_3-L_{A_\alpha}>0$, $\mu_1-L_{\Phi_\alpha}\lambda_1>0$ can be similar discussed\big), from (4.5) we have
\begin{equation}\label{equ 2018103101}
\begin{aligned}
&\ (\mu_1-L_{\Phi_\alpha}\lambda_1)E|\widehat{x}(T)|^2+(\beta_1-C_0L_{A_\alpha})E\int_0^T|\widehat{x}(t)|^2dt\\
\leq &\ C_2E\[E'|\widehat{\Phi}(T)|^2+\int_0^T\int_EE'\(|\widehat{b}(t,e)|^2+|\widehat{\sigma}(t,e)|^2
+|\widehat{h}(t,e)|^2+|\widehat{f}(t,e)|^2\)
\lambda(de)dt\]\\
&\ +\delta\[E|\widehat{x}(T)|^2+E\int_0^T\(|\widehat{x}(t)|^2+|\widehat{y}(t)|^2+|\widehat{z}(t)|^2
+\int_E|\widehat{k}(t,e)|^2\lambda(de)\)dt\].
\end{aligned}\end{equation}
Using (\ref{equ 2018103102}) and (\ref{equ 2018103101}) we can take sufficiently small $\delta$ such that
\begin{equation}
\begin{aligned}
&\ E|\widehat{x}(T)|^2+E\int_0^T\(|\widehat{x}(t)|^2+|\widehat{y}(t)|^2+|\widehat{z}|^2
+\int_E|\widehat{k}(t,e)|^2\lambda(de)\)dt\\
\leq&\  \ C\bar{E}\[|\widehat{\Phi}(T)|^2 +\int_0^T\int_E\(|\widehat{b}(t,e)|^2+|\widehat{\sigma}(t,e)|^2+|\widehat{h}(t,e)|^2+|\widehat{f}(t,e)|^2\)dt\],
\end{aligned}
\end{equation}
here the constant $C$ only depends on $C_1,\,C_2,\,\beta_1,\,\mu_1,\ L_{A_\alpha},\ L_{\Phi_\alpha}$.\\
%Similarly, when $\beta_1-C_0L_A=0$, $\beta_2-C_0L_A>0$, $\beta_3-L_A>0$, $\mu_1-L_\Phi\lambda_1>0$, then from (4.5) we have
%\begin{equation}\label{equ 2018103104}
%\begin{aligned}
% &\ (\beta_2-C_0L_A)E\int_0^T\(|\widehat{y}(t)|^2+|\widehat{z}(t)|^2\)dt
% +(\beta_3-L_A)E\int_0^T\int_E|\widehat{k}(t,e)|^2\lambda(de)dt\\
%\leq &\  C_2E\[E'|\widehat{\Phi}(T)|^2+\int_0^T\int_EE'\(|\widehat{b}(t,e)|^2+|\widehat{\sigma}(t,e)|^2
%+|\widehat{h}(t,e)|^2+|\widehat{f}(t,e)|^2\)dt\]\\
% &\ +\delta\[E|\widehat{x}(T)|^2+E\int_0^T\(|\widehat{x}(t)|^2+|\widehat{y}(t)|^2+|\widehat{z}(t)|^2
% +\int_E|\widehat{k}(t,e)|^2\lambda(de)\)dt\].
%\end{aligned}
%\end{equation}
%With the help of (4.3) and (4.8) we can take sufficiently small $\delta$ such that
%\begin{equation}
%\begin{aligned}
%&\ E|\widehat{x}(T)|^2+E\int_0^T\(|\widehat{x}(t)|^2+|\widehat{y}(t)|^2+|\widehat{z}(t)|^2
%+\int_E|\widehat{k}(t,e)|^2\lambda(de)\)dt\\
%\leq&\   C\bar{E}\[|\widehat{\Phi}(T)|^2 +\int_0^T\int_E\(|\widehat{b}(t,e)|^2+|\widehat{\sigma}(t,e)|^2+|\widehat{h}(t,e)|^2+|\widehat{f}(t,e)|^2\)dt\],
%\end{aligned}
%\end{equation}
%here the constant $C$ only depends on $C_1,\,C_2,\,\beta_2,\ \beta_3,\ C_0,\ L_A,\ L_\Phi$.\\
Hence, we have that $(x^\alpha, y^\alpha, z^\alpha,k^\alpha, x^{\alpha}(T))$ converges to $(x^0, y^0, z^0,k^0, x^0(T))$ in $\mathcal{H}_{\mathbb{F}}^2(0,T;\mathbb{R}^n\times\mathbb{R}^m\times\mathbb{R}^{m\times d})
\times\mathcal{K}^2_{\mathbb{F},\lambda}(0,T;\mathbb{R}^m)
\times L^2(\Omega,\mathcal{F}_T,P;\mathbb{R}^n)$
as $\alpha$ tends to 0.
\end{proof}

\section{Maximum principle for the controlled fully coupled mean-field FBSDEs with jumps}
We consider the following controlled fully coupled mean-field forward-backward SDEs with jumps:
\begin{equation}\label{Eq6.1}
\left\{
\begin{aligned}
dx^v(t)=&\ \int_EE'[b\big(t,\pi^v(t,e),v(t)\big)]\lambda(de)dt +\int_EE'[\sigma\big(t,\pi^v(t,e),v(t)\big)]\lambda(de)dB_t\\
&\ +\int_EE'[h\big(t,\pi^v(t-,e),v(t),e\big)]\widetilde{\mu}(dtde),\\
-dy^v(t)=&\ \int_EE'[f\big(t,\pi^v(t,e),v(t)\big)]\lambda(de)dt-z^v(t)dB_t-\int_Ek^v(t,e)\widetilde{\mu}(dtde),\\
x^v(0)=&\ a,\ \ y^v(T)=E'[\Phi\big(x^v(T),(x^v(T))'\big)],
\end{aligned}
\right.
\end{equation}
where $$\pi^v(t,e)=\big(x^v(t), y^v(t), z^v(t), k^v(t,e), (x^v(t))',(y^v(t))',(z^v(t))', (k^v(t,e))'\big),$$
$$\pi^v(t-,e)=\big(x^v(t-), y^v(t-), z^v(t), k^v(t,e), (x^v(t-))',(y^v(t-))', (z^v(t))', (k^v(t,e))'\big).$$ Let $U$\ be a nonempty convex subset of $\mathbb{R}^k$, we define the admissible control set
$$\mathcal{U}_{ad}=\{v(\cdot)\in \mathcal{H}_{\bar{\mathbb{F}}}^2(0, T; \mathbb{R}^k)|v(t)\in U,\ 0\leq t\leq T,\ \bar{P}\mbox{-a.s.}\}.$$
We now define the following cost functional:
\begin{equation}\label{equ 2018012301}
\begin{aligned}
 J(v(\cdot))=E\Big[ & \displaystyle \int_0^T\int_EE'\big[g\big(t,\pi^v(t,e),v(t)\big)\big]\lambda(de)dt+E'[\varphi\big(x^v(T),(x^v(T))'\big)]+\gamma\big(y^v(0)\big)\Big],
\end{aligned}
\end{equation}
where
$$\begin{array}{lll}
& &b:[0,T]\times \mathbb{R}^n \times \mathbb{R}^m \times \mathbb{R}^{m\times d}\times \mathbb{R}^m \times \mathbb{R}^n \times \mathbb{R}^m \times \mathbb{R}^{m\times d}\times \mathbb{R}^m \times U\rightarrow \mathbb{R}^n,\\
& &\sigma:[0,T]\times \mathbb{R}^n \times \mathbb{R}^m \times \mathbb{R}^{m\times d}\times \mathbb{R}^m \times \mathbb{R}^n \times \mathbb{R}^m \times \mathbb{R}^{m\times d}\times \mathbb{R}^m \times U\rightarrow \mathbb{R}^{n\times d},\\
& &h:[0,T]\times \mathbb{R}^n \times \mathbb{R}^m \times \mathbb{R}^{m\times d}\times \mathbb{R}^m \times \mathbb{R}^n \times \mathbb{R}^m \times \mathbb{R}^{m\times d}\times \mathbb{R}^m \times U\times E\rightarrow \mathbb{R}^{n},\\
& &f:[0,T]\times \mathbb{R}^n \times \mathbb{R}^m \times \mathbb{R}^{m\times d}\times \mathbb{R}^m\times \mathbb{R}^n \times \mathbb{R}^m \times \mathbb{R}^{m\times d}\times \mathbb{R}^m\times U\rightarrow \mathbb{R}^{m},\\
& &g:[0,T]\times \mathbb{R}^n \times \mathbb{R}^m \times \mathbb{R}^{m\times d}\times \mathbb{R}^m\times \mathbb{R}^n \times \mathbb{R}^m \times \mathbb{R}^{m\times d}\times \mathbb{R}^m\times U\rightarrow \mathbb{R},\\
& &\Phi:\mathbb{R}^n\times\mathbb{R}^n\rightarrow \mathbb{R}^m \ , \ \varphi:\mathbb{R}^n\times\mathbb{R}^n\rightarrow \mathbb{R} \ , \ \gamma:\mathbb{R}^m\rightarrow \mathbb{R}.\\
\end{array}
$$
Our stochastic optimal control problem is to minimize the cost functional $J(v(\cdot))$ over all admissible controls. An admissible control $u(\cdot)$ is called an optimal control if the cost functional $J(v(\cdot))$ attains the minimum at $u(\cdot)$.
Equation (\ref{Eq6.1}) is called the state equation, the solution $(x(\cdot),y(\cdot),z(\cdot),k(\cdot,\cdot))$ corresponding to $u(\cdot)$ is called the optimal trajectory.\\
We assume
\begin{eqnarray*}{\rm(H5.1)}
\left\{ \begin{array}{lllll}
\rm{(i)}& b,\  \sigma,\ h,\ f,\ g,\ \Phi,\ \varphi\ \mbox{and}\ \gamma\ \mbox{are continuously differentiable to } (x,\ y,\ z,\ k,\ \widetilde{x},\ \widetilde{y},\ \widetilde{z},\ \widetilde{k},\ v);\\
\rm{(ii)}& \mbox{The derivatives of } b,\ \sigma,\ h,\ f,\ \Phi\ \mbox{are bounded};\\
\rm{(iii)}& \mbox{The derivatives of $g$ are bounded by } C(1+|x|+|y|+|z|+|k|+|\widetilde{x}|+|\widetilde{y}|+|\widetilde{z}|+|\widetilde{k}|+|v|) ;\\
\rm{(iv)}& \mbox{The derivatives of}\  \varphi\ \mbox{and}\ \gamma\ \mbox{are bounded by } C(1+|x|+|\widetilde{x}|)\ \mbox{and}\ C(1+|y|),\ \mbox{respectively} ;\\
 \rm{(v)}& \mbox{For any given admissible control}\ v(\cdot),\ \mbox{the coefficients satisfy (H3.1) and (H3.2)}.
\end{array}\right.
\end{eqnarray*}

Let $u(\cdot)$ be an optimal control and $(x(\cdot), y(\cdot), z(\cdot),k(\cdot,\cdot))$ be the corresponding optimal trajectory.
Let $v(\cdot)$ be such that $u(\cdot)+v(\cdot)\in\mathcal{U}_{ad}$. Since $U$ is convex, we may choose the perturbation
  $$u_\rho(\cdot)=u(\cdot)+\rho v(\cdot)\in\mathcal{U}_{ad},$$
  for any $0\leq\rho\leq1$.

To simplify the form of the following variational equation (\ref{Eq6.2}), variational inequality (\ref{Eq6.4}) and adjoint equation (\ref{Eq6.5}), we introduce the following notations:
\begin{eqnarray*}
\begin{split}
&\theta(t,e)=\big(t, x(t), y(t), z(t), k(t,e), (x(t))', (y(t))', (z(t))', (k(t,e))', u(t)\big),\\
&\theta(t-,e)=\big(t, x(t-), y(t-), z(t), k(t,e), (x(t-))', (y(t-))', (z(t))', (k(t,e))', u(t), e\big),\\
&\rho(t,e)=\big(t, (x(t))', (y(t))', (z(t))', (k(t,e))', x(t), y(t), z(t), k(t,e), (u(t))'\big),\\
&\rho(t-,e)=\big(t, (x(t-))', (y(t-))', (z(t))', (k(t,e))', x(t-), y(t-), z(t), k(t,e), (u(t))', e\big).
\end{split}
\end{eqnarray*}

We denote by $(x_\rho(\cdot),\ y_\rho(\cdot),\ z_\rho(\cdot),k_\rho(\cdot,\cdot))$\ the trajectory corresponding to $u_\rho$. Then we have the following convergence result.
\begin{lemma}\label{lem6.1}
Under the assumption (H5.1), it holds
$$\lim_{\rho\rightarrow0}\frac{x_\rho(t)-x(t)}{\rho}=x^1(t), \ \ \lim_{\rho\rightarrow0}\frac{y_\rho(t)-y(t)}{\rho}=y^1(t),\  \lim_{\rho\rightarrow0}\frac{z_\rho(t)-z(t)}{\rho}=z^1(t),\ \text{in}\ \mathcal{H}_{\mathbb{F}}^2(0,T).$$
$$\lim_{\rho\rightarrow0}\frac{k_\rho(t,e)-k(t,e)}{\rho}=k^1(t,e),\  \text{in}\ \mathcal{K}_{\mathbb{F}\lambda}^2(0,T),$$
where $(x^1(\cdot), y^1(\cdot), z^1(\cdot), k^1(\cdot,\cdot))$ is the unique solution of  the following variational equation:
\begin{equation}\label{Eq6.2}
\left\{
\begin{aligned}
dx^1(t)=&\int_E E'\big\{b_x(\theta(t,e))x^1(t)+b_y(\theta(t,e))y^1(t)+b_z(\theta(t,e))z^1(t)+b_k(\theta(t,e))k^1(t,e)+b_v(\theta(t,e))v(t)\\
&+b_{\widetilde{x}}(\theta(t,e))(x^1(t))'
+b_{\widetilde{y}}(\theta(t,e))(y^1(t))'+b_{\widetilde{z}}(\theta(t,e))(z^1(t))'+b_{\widetilde{k}}(\theta(t,e))(k^1(t,e))'\big\}\lambda(de)dt\\
&+\int_E E'\big\{\sigma_x(\theta(t,e))x^1(t)+\sigma_y(\theta(t,e))y^1(t)+\sigma_z(\theta(t,e))z^1(t)+\sigma_k(\theta(t,e))k^1(t,e)+\sigma_v(\theta(t,e))v(t)\\
&+\sigma_{\widetilde{x}}(\theta(t,e))(x^1(t))'
+\sigma_{\widetilde{y}}(\theta(t,e))(y^1(t))'+\sigma_{\widetilde{z}}(\theta(t,e))(z^1(t))'+\sigma_{\widetilde{k}}(\theta(t,e))(k^1(t,e))'\big\}\lambda(de)dB_t\\
&+\int_E E'\big\{h_x(\theta(t-,e))x^1(t)+h_y(\theta(t-,e))y^1(t)+h_z(\theta(t-,e))z^1(t)+h_k(\theta(t-,e))k^1(t,e)\\
&+h_v(\theta(t-,e))v(t)+h_{\widetilde{x}}(\theta(t-,e))(x^1(t))'+h_{\widetilde{y}}(\theta(t-,e))(y^1(t))'+h_{\widetilde{z}}(\theta(t-,e))(z^1(t))'\\
&+h_{\widetilde{k}}(\theta(t-,e))(k^1(t,e))'\big\}\widetilde{\mu}(dtde),\\
-dy^1(t)=&\ \int_E E'\big\{f_x(\theta(t,e))x^1(t)+f_y(\theta(t,e))y^1(t)+f_z(\theta(t,e))z^1(t)+f_k(\theta(t,e))k^1(t,e)+f_v(\theta(t,e))v(t)\\
&+f_{\widetilde{x}}(\theta(t,e))(x^1(t))'+f_{\widetilde{y}}(\theta(t,e))(y^1(t))'+f_{\widetilde{z}}(\theta(t,e))(z^1(t))'+f_{\widetilde{k}}(\theta(t,e))(k^1(t,e))'\big\}\lambda(de)dt\\
&
-z^1(t)dB_t-\int_Ek^1(t,e)\widetilde{\mu}(dtde),\\
x^1(0)=&0,\ y^1(T)=E'[\Phi_x\big(x(T), (x(T))'\big)x^1(T)+\Phi_{\widetilde{x}}\big(x(T),(x(T))'\big)(x^1(T))'].
\end{aligned}
\right.
\end{equation}
\end{lemma}
\begin{remark}\label{rem5.1}
(i)When $l=b,\sigma, h,f,\Phi$, respectively, $l_x$ is the partial derivative of $l(t,x,y,z,\widetilde{x},\widetilde{y},\widetilde{z},k,\widetilde{k},v)$ with respect to $x$; $l_{\widetilde{x}}$ is the partial derivative of $l(t,x,y,z,\widetilde{x},\widetilde{y},\widetilde{z},k,\widetilde{k},v)$ with respect to $\widetilde{x}$. Similar to $l_y,l_z,l_k,l_{\widetilde{y}},l_{\widetilde{z}},l_{\widetilde{k}},l_v$.\\
(ii)From (H5.1), it is easy to verify that equation (\ref{Eq6.2}) satisfies (H3.1) and (H3.2), then there exists a unique solution $(x^1, y^1, z^1,k^1)$ of linear mean-field FBSDE (\ref{Eq6.2}).
\end{remark}
\begin{proof}
Let $\widehat{x}(t)=x_\rho(t)-x(t), \  \widehat{y}(t)=y_\rho(t)-y(t), \ \widehat{z}(t)=z_\rho(t)-z(t), \ \widehat{k}(t,e)=k_\rho(t,e)-k(t,e)$.
Then
\begin{equation}\label{equ 2017121901}
\left\{
\begin{aligned}
\ d\widehat{x}(t)=&\int_E E'\Big[b\big(t,\pi_\rho(t,e),u_\rho(t)\big)-b\big(t,\pi(t,e),u(t)\big)\Big]\lambda(de)dt\\
 +&\int_EE'\Big[\sigma\big(t,\pi_\rho(t,e),u_\rho(t)\big)-\sigma\big(t,\pi(t,e),u(t)\big)\Big]\lambda(de)dB_t\\
\ +&\int_EE'\Big[h\big(t,\pi_\rho(t-,e),u_\rho(t),e\big)
-h\big(t,\pi(t-,e),u(t),e\big)\Big]\widetilde{\mu}(dtde),\\
-d\widehat{y}(t)=&\ \int_EE'\Big[f\big(t,\pi_\rho(t,e),u_\rho(t))-f\big(t,\pi(t,e),u(t)\big)\Big]\lambda(de)dt\\
&-\widehat{z}(t)dB_t-\int_E\widehat{k}(t,e)
\widetilde{\mu}(dtde),\\
\widehat{x}(0)=&\ 0, \ \widehat{y}(T)=E'\Big[\Phi\big(x_\rho(T),(x_\rho(T))'\big)-\Phi\big(x(T),(x(T))'\big)\Big],
\end{aligned}
\right.
\end{equation}
where $\pi(t,e)=\pi^u(t,e)$, $\pi_\rho(t,e)=\pi^{u_\rho}(t,e)$, $\pi(t-,e)$ and $\pi_\rho(t-,e)$ are similarly defined.
From Theorem \ref{Thm4.1}, it is easy to know that $\big(\widehat{x}(\cdot), \widehat{y}(\cdot), \widehat{z}(\cdot), \widehat{k}(\cdot,\cdot)\big)$\ converges to $0$ in $\big(M_{\mathbb{F}}^2(0,T)\big)^3\times\mathcal{K}_{\mathbb{F}\lambda}^2(0,T)$  as $\rho$ tends to $0$. Now, we define
$
\Delta l(t)=\frac{l_\rho(t)-l(t)}{\rho},\,\, l=x,y,z,\ \
\Delta k(t,e)=\frac{k_\rho(t,e)-k(t,e)}{\rho}.
$
Then, from (\ref{equ 2017121901}) we have
\begin{equation}
\left\{
\begin{aligned}
\ d\Delta{x}(t)=&\int_E E'\Big[\bar{b}\big(t,\Delta\pi(t,e),v(t)\big)\Big]\lambda(de)dt +\int_EE'\Big[\bar{\sigma}\big(t,\Delta\pi(t),v(t)\big)\Big]\lambda(de)dB_t\\
\ +&\int_EE'\Big[\bar{h}\big(t,\Delta\pi(t-,e),v(t),e\big)\Big]\widetilde{\mu}(dtde),\\
-d\Delta{y}(t)=&\ \int_EE'\Big[\bar{f}\big(t,\Delta\pi(t,e),v(t)\big)\Big]\lambda(de)dt
-\Delta{z}(t)dB_t-\int_E\Delta{k}(t,e)
\widetilde{\mu}(dtde),\\
\Delta{x}(0)=&\ 0, \ \Delta{y}(T)=E'\Big[M(T)\Delta x(T)+N(T)(\Delta x(T))'\Big],
\end{aligned}
\right.
\end{equation}
where $$\Delta \pi(t,e)=(\Delta x(t),\ \Delta y(t),\ \Delta z(t),\ \Delta k(t,e),\ (\Delta x(t))',\ (\Delta y(t))',\ (\Delta z(t))',\ (\Delta k(t,e))'),$$
$$\Delta \pi(t-,e)=(\Delta x(t-),\ \Delta y(t-),\ \Delta z(t),\ \Delta k(t,e),\ (\Delta x(t-))',\ (\Delta y(t-))',\ (\Delta z(t))',\ (\Delta k(t,e))'),$$
\begin{equation}\nonumber
\begin{aligned}
&\bar{b}(t,\ x,\ y,\ z,\ k,\ \widetilde{x},\ \widetilde{y},\ \widetilde{z},\ \widetilde{k},\  v)\\
=&A(t,e)x+B(t,e)y+C(t,e)z+D(t,e)k+E(t,e)\widetilde{x}+F(t,e)\widetilde{y}+G(t,e)\widetilde{z}+H(t,e)\widetilde{k}+I(t,e)v,\\
\end{aligned}
\end{equation}
and
\begin{equation*}
\begin{aligned}
 &A(t,e)\Delta x(t)=\frac{1}{\rho}\Big[b\big(t,x_\rho(t),y_\rho(t),\cdots,u_\rho(t) \big)-b\big(t,x(t),y_\rho(t),\cdots, u_\rho(t)\big)\Big],\\
&B(t,e)\Delta y(t)=\frac{1}{\rho}\Big[b\big(t,x(t),y_\rho(t),z_\rho(t),\cdots,u_\rho(t) \big)-b\big(t,x(t),y(t),z_\rho(t),\cdots,u_\rho(t)\big)\Big],\\
&C(t,e)\Delta z(t)=\frac{1}{\rho}\Big[b\big(t,x(t),y(t),z_\rho(t),k_\rho(t,e),\cdots,u_\rho(t) \big)-b\big(t,x(t),y(t),z(t),k_\rho(t,e),\cdots,u_\rho(t)\big)\Big],\\
&D(t,e)\Delta k(t,e)=\frac{1}{\rho}\Big[b\big(t,\cdots,z(t),k_\rho(t,e),(x_\rho(t))',\cdots \big)-b\big(t,\cdots,z(t),k(t,e),(x_\rho(t))',\cdots\big)\Big],\\
 &E(t,e)(\Delta x(t))'=\frac{1}{\rho}\Big[b\big(t,\cdots,k(t,e),(x_\rho(t))',(y_\rho(t))',\cdots \big)-b\big(t,\cdots,k(t,e),(x(t))',(y_\rho(t))',\cdots\big)\Big],\\
 &F(t,e)(\Delta y(t))'=\frac{1}{\rho}\Big[b\big(t,\cdots,(x(t))',(y_\rho(t))',(z_\rho(t))',\cdots \big)-b\big(t,\cdots,(x(t))',(y(t))',(z_\rho(t))',\cdots\big)\Big],\\
 &G(t,e)(\Delta z(t))'=\frac{1}{\rho}\Big[b\big(t,\cdots,(y(t))',(z_\rho(t))',(k_\rho(t,e))',u_\rho(t) \big)-b\big(t,\cdots,(y(t))',(z(t))',(k_\rho(t,e))',u_\rho(t)\big)\Big],\\
 &H(t,e)(\Delta k(t,e))'=\frac{1}{\rho}\Big[b\big(t,\cdots,(z(t))',(k_\rho(t,e))',u_\rho(t) \big)-b\big(t,\cdots,(z(t))',(k(t,e))',u_\rho(t)\big)\Big],\\
\end{aligned}
\end{equation*}

\begin{equation*}
\begin{aligned}
 &I(t,e)v(t)=\frac{1}{\rho}\Big[b\big(t,\cdots,(k(t,e))',u_\rho(t)\big)-b\big(t,\cdots,(k(t,e))',u(t)\big)\Big],\\
  &M(T)\Delta x(T)=\frac{1}{\rho}\Big[\Phi\big(x_\rho(T),(x_\rho(T))'\big)-\Phi\big(x(T),(x_\rho(T))'\big)\Big],\\
 &N(T)(\Delta x(T))'=\frac{1}{\rho}\Big[\Phi\big(x(T),(x_\rho(T))'\big)-\Phi\big(x(T),(x(T))'\big)\Big],\\
\end{aligned}
\end{equation*}
where $\bar{\sigma}$, $\bar{h}$, $\bar{f}$ are similarly defined.
From (H5.1) and the fact $(\widehat{x}(\cdot),\widehat{y}(\cdot),\widehat{z}(\cdot),
\widehat{k}(\cdot,\cdot))$ converges to 0 in $\big(\mathcal{H}_{\mathbb{F}}^2(0,T)\big)^3\times\mathcal{K}_{\mathbb{F}\lambda}^2(0,T)$ as $\rho$ tends to 0, we know
\begin{eqnarray*}
\begin{split}
&\lim_{\rho \rightarrow 0} [A(t,e)-b_x\big(\theta(t,e)\big)]=0, \ \ \lim_{\rho \rightarrow 0}[B(t,e)-b_y\big(\theta(t,e)\big)]=0,\ \ \lim_{\rho \rightarrow 0} [C(t,e)-b_z\big(\theta(t,e)\big)]=0,  \\
  & \lim_{\rho \rightarrow 0} [D(t,e)-b_{k}\big(\theta(t,e)\big)]=0,\ \ \lim_{\rho \rightarrow 0} [E(t,e)-b_{\widetilde{x}}\big(\theta(t,e)\big)]=0, \ \ \lim_{\rho \rightarrow 0} [F(t,e)-b_{\widetilde{y}}\big(\theta(t,e)\big)]=0, \\
  &\lim_{\rho \rightarrow 0} [G(t,e)-b_{\widetilde{z}}\big(\theta(t,e)\big)]=0, \ \ \lim_{\rho \rightarrow 0} [H(t,e)-b_{\widetilde{k}}\big(\theta(t,e)\big)]=0,\ \ \lim_{\rho \rightarrow 0} [I(t,e)-b_v\big(\theta(t,e)\big)]=0,
\end{split}
\end{eqnarray*}
and
\begin{equation*}
\begin{aligned}
&\lim_{\rho \rightarrow 0}\Big\{\bar{b}\big(t,\Delta\pi(t,e),v(t)\big)-b_x\big(\theta(t,e)\big)\Delta x(t)-b_y\big(\theta(t,e)\big)\Delta y(t)-b_z\big(\theta(t,e)\big)\Delta z(t)-b_k\big(\theta(t,e)\big)\Delta k(t,e)\\
 &-b_{\widetilde{x}}\big(\theta(t,e)\big)(\Delta x(t))'-b_{\widetilde{y}}\big(\theta(t,e)\big)(\Delta y(t))'-b_{\widetilde{z}}\big(\theta(t,e)\big)(\Delta z(t))'-b_{\widetilde{k}}\big(\theta(t,e)\big)(\Delta k(t,e))'-b_v\big(\theta(t,e)\big)v(t)\Big\}=0,
 \end{aligned}
\end{equation*}
$\bar{\sigma},\ \bar{h},\ \bar{f}$, $\Delta y(T)$ have similar results. From the uniqueness of the solution of equation (\ref{Eq6.2}), we know  $(\Delta x(\cdot),\ \Delta y(\cdot),\ \Delta z(\cdot),\ \Delta k(\cdot,\cdot))$\ converges to $(x^1(\cdot),\ y^1(\cdot),\ z^1(\cdot),\ k^1(\cdot,\cdot))$ in $\big(\mathcal{H}_{\mathbb{F}}^2(0,T)\big)^3\times\mathcal{K}_{\mathbb{F}\lambda}^2(0,T)$ as $\rho$ tends to $0$.
\end{proof}

Because $u(\cdot)$\ is an optimal control, then
\begin{equation}\label{Eq6.3}
\rho^{-1}[J(u(\cdot)+\rho v(\cdot))-J(u(\cdot))]\geq 0.
\end{equation}
Using the similar approach of Lemma \ref{lem6.1}, from (\ref{Eq6.3}) we have the following results.
\begin{lemma}\label{lem6.2}
We suppose (H5.1) holds. Then, the following variational inequality holds:
\begin{equation}\label{Eq6.4}
\begin{aligned}
&E\Big\{\int_0^T\int_EE'[g_x\big(\theta(t,e)\big)x^1(t)+g_y\big(\theta(t,e)\big)y^1(t)+g_z\big(\theta(t,e)\big)z^1(t)+g_k\big(\theta(t,e)\big)k^1(t,e)
+g_{\widetilde{x}}\big(\theta(t,e)\big)(x^1(t))'\\
&\ \ +g_{\widetilde{y}}\big(\theta(t,e)\big)(y^1(t))'+g_{\widetilde{z}}\big(\theta(t,e)\big)(z^1(t))'+g_{\widetilde{k}}\big(\theta(t,e)\big)(k^1(t,e))' +g_v\big(\theta(t,e)\big)v(t)]\lambda(de)dt \\
      &\ \ +E'[\varphi_x\big(x(T),(x(T))'\big)x^1(T)+\varphi_{\widetilde{x}}\big(x(T),(x(T))'\big)(x^1(T))']+\gamma_y(y(0))y^1(0)\Big\}\geq 0.
\end{aligned}
\end{equation}
\end{lemma}
%\begin{proof}
%Let $\rho\rightarrow0$ in (\ref{Eq6.3}), from Lemma \ref{lem6.1} and (H5.1), it is obvious that
%\begin{equation*}
%\begin{aligned}
% &\rho^{-1}E(E'[\varphi(x_{\rho}(T)),(x_{\rho}(T)')-\varphi(x(T),(x(T))')])\\
%\rightarrow &EE'[\varphi_x(x(T),(x(T))')x^1(T)+\varphi_{\widetilde{x}}(x(T),(x(T))')(x^1(T))'];\\
% &\rho^{-1}E[h(y_{\rho}(0))-h(y(0))]\rightarrow E[h_y(y(0))y^1(0)];\\
%&\rho^{-1}E\int_0^TE'[L(\chi_\rho(t),u(t)+\rho v(t))-L(\theta(t))]dt\\
%\rightarrow &E\int_0^TE'[L_x(\theta(t))x^1(t)+L_y(\theta(t))y^1(t)+L_z(\theta(t))z^1(t)\\
%&+L_{\widetilde{x}}(\theta(t))(x^1(t))'+L_{\widetilde{y}}(\theta(t))(y^1(t))'+L_{\widetilde{z}}(\theta(t))(z^1(t))'+L_v(\theta(t))v(t)]dt.
 %\end{aligned}
%\end{equation*}
%The proof is complete.
%\end{proof}
Now we introduce the following adjoint mean-field FBSDE with jumps to equation (\ref{Eq6.2}):
\begin{equation}\label{Eq6.5}
\left\{
\begin{aligned}
dp(t)=&-\int_E E' \Big\{b^T_y(\theta(t,e))q(t)+\sigma^T_y(\theta(t,e))m(t)+h_y^T\big(\theta(t-,e)\big)n(t,e)-f_y^T\big(\theta(t,e)\big)p(t)\\
&\ +g_y\big(\theta(t,e)\big)+b^T_{\widetilde{y}}(\rho(t,e))(q(t))' +\sigma^T_{\widetilde{y}}(\rho(t,e))(m(t))' +h_{\widetilde{y}}^T\big(\rho(t-,e)\big)(n(t,e))'\\
&\ -f_{\widetilde{y}}^T\big(\rho(t,e)\big)(p(t))'+g_{\widetilde{y}}\big(\rho(t,e)\big)\Big\}\lambda(de)dt\\
&\ -\int_E E'\Big\{b^T_z(\theta(t,e))q(t)+\sigma^T_z(\theta(t,e))m(t)+h_z^T\big(\theta(t-,e)\big)n(t,e)-f_z^T\big(\theta(t,e)\big)p(t)\\
&\ +g_z\big(\theta(t,e)\big)+b^T_{\widetilde{z}}(\rho(t,e))(q(t))'+\sigma^T_{\widetilde{z}}(\rho(t,e))(m(t))'+h_{\widetilde{z}}^T\big(\rho(t-,e)\big)(n(t,e))'\\
&\ -f_{\widetilde{z}}^T\big(\rho(t,e)\big)(p(t))'+g_{\widetilde{z}}\big(\rho(t,e)\big)\Big\}\lambda(de)dB_t\\
&\ -\int_E E'\Big\{b^T_k(\theta(t,e))q(t)+\sigma^T_k(\theta(t,e))m(t)+h_k^T\big(\theta(t-,e)\big)n(t,e)-f_k^T\big(\theta(t,e)\big)p(t)\\
&\ +g_k\big(\theta(t,e)\big)+b^T_{\widetilde{k}}(\rho(t,e))(q(t))'+\sigma^T_{\widetilde{k}}(\rho(t,e))(m(t))'+h_{\widetilde{k}}^T\big(\rho(t-,e)\big)(n(t,e))'\\
&\ -f_{\widetilde{k}}^T\big(\rho(t,e)\big)(p(t))'+g_{\widetilde{k}}\big(\rho(t,e)\big)\Big\}\widetilde{\mu}(dtde),\\
-dq(t)=&\ \int_E E'\Big\{b^T_x(\theta(t,e))q(t)+\sigma^T_x(\theta(t,e))m(t)+h_x^T\big(\theta(t-,e)\big)n(t,e)-f_x^T\big(\theta(t,e)\big)p(t)\\
&\ +g_x\big(\theta(t,e)\big)+b^T_{\widetilde{x}}(\rho(t,e))(q(t))'+\sigma^T_{\widetilde{x}}(\rho(t,e))(m(t))'+h_{\widetilde{x}}^T\big(\rho(t-,e)\big)(n(t,e))'\\
&\ -f_{\widetilde{x}}^T\big(\rho(t,e)\big)(p(t))'+g_{\widetilde{x}}\big(\rho(t,e)\big)\Big\}\lambda(de)dt-m(t)dB_t-\int_En(t,e)\widetilde{\mu}(dtde),\\
p(0)=&\ -\gamma_y(y(0)),\\
 q(T)=&E'[\varphi_x\big(x(T),(x(T))'\big)+\varphi_{\widetilde{x}}\big((x(T))',x(T)\big)\\
 &-\Phi_x\big(x(T),(x(T))'\big)p(T)-\Phi_{\widetilde{x}}\big((x(T))',x(T)\big)(p(T))'].
\end{aligned}
\right.
\end{equation}

  From Theorem 3.1, we know there exists a unique quadruple $(p(\cdot), q(\cdot), m(\cdot), n(\cdot,\cdot))$\ satisfying (\ref{Eq6.5}).\\
 We define the Hamiltonian function $H$ as follows:
 \begin{equation}\label{Eq6.6}
\begin{aligned}
& H(t,x,y,z,k,\widetilde{x},\widetilde{y},\widetilde{z},\widetilde{k},v,p,q,m,n,e)=\langle q,b(t,x,y,z,k,\widetilde{x},\widetilde{y},\widetilde{z},\widetilde{k},v)\rangle+\langle m,\sigma(t,x,y,z,k,\widetilde{x},\widetilde{y},\widetilde{z},\widetilde{k},v)\rangle\\
&+\langle n,h(t,x,y,z,k,\widetilde{x},\widetilde{y},\widetilde{z},\widetilde{k},v,e)\rangle-\langle p,f(t,x,y,z,k,\widetilde{x},\widetilde{y},\widetilde{z},\widetilde{k},v)\rangle
+g(t,x,y,z,k,\widetilde{x},\widetilde{y},\widetilde{z},\widetilde{k},v).
 \end{aligned}
\end{equation}
 Then we have the following maximum principle.
 \begin{theorem}\label{Thm6.1}
Let $u(\cdot)$\ be an optimal control and let $(x(\cdot),\ y(\cdot),\ z(\cdot),\ k(\cdot,\cdot))$ be the corresponding trajectory. Then, we have
 \begin{equation}\label{smp}
\begin{aligned}
\int_EE'\big\langle H_v\big(t,\pi(t,e),\ u(t),\ p(t),\ q(t),\ m(t),\ n(t,e),\ e\big),\ v-u(t)\big\rangle\lambda(de)\geq0, \forall\ v\in U, \mbox{dtdP-a.e.},
 \end{aligned}
\end{equation}
where $\pi(t,e)=(x(t), y(t), z(t), k(t,e), (x(t))', (y(t))', (z(t))', (k(t,e))')$, $(p(\cdot), q(\cdot), m(\cdot), n(\cdot,\cdot))$\ is the solution of the adjoint equation (\ref{Eq6.5}).
\end{theorem}

\begin{proof}
Applying It\^o's formula to $\langle x^1(t),q(t)\rangle+\langle y^1(t),p(t)\rangle$, from equations (\ref{Eq6.2}) and (\ref{Eq6.5}), (H3.1),
 (H3.2) and (H5.1), with the help of (\ref{Eq6.4}) and (\ref{Eq6.6}), for $v(\cdot)$ such that $u(\cdot)+v(\cdot)\in\mathcal{U}_{ad}$, we get
\begin{equation}\label{equ 2017122101}
E\int_0^T\int_EE'\langle H_v\big(t,\pi(t,e),\ u(t),\ p(t),\ q(t),\ m(t),\ n(t,e),\ e\big), v(t)\rangle\lambda(de) dt\geq0.
\end{equation}
Denote $H_v(t,e)=H_v\big(t,\pi(t,e),\ u(t),\ p(t),\ q(t),\ m(t),\ n(t,e),\ e\big)$. For any $\bar{v}(\cdot)\in\mathcal{U}_{ad}$, we define
$$v(s)=
\left\{
\begin{array}{ll}
\bar{v}(s)-u(s),& s\in [t,t+\varepsilon],\\
0,& \text{otherwise}.
\end{array}
\right.
$$
Then from (\ref{equ 2017122101}) we get
\begin{equation}\label{equ 2017122102}
\frac{1}{\varepsilon}E\int_t^{t+\varepsilon}\int_EE'\langle H_v(s,e), \bar{v}(s)-u(s)\rangle\lambda(de) dt\geq0.
\end{equation}
Putting $\varepsilon\rightarrow 0$, we have ${E}\int_EE'\langle H_v(t,e), \bar{v}(t)-u(t)\rangle\lambda(de)\geq 0$, $a.e.$
Then, let $\bar{v}(t)=vI_A+u(t)I_{A^c}$, for $A\in\mathcal{F}_t$ and $v\in U$, we can get that
 \begin{equation}
\begin{aligned}
0\leq& {E}\int_EE'\langle H_v(t,e), \bar{v}(t)-u(t)\rangle\lambda(de)=E[\int_EE'\langle H_v(t,e), v-u(t)\rangle \lambda(de)\cdot I_A]\\
=&E[\int_EE'\langle H_v(t,e), v-u(t)\rangle \lambda(de)| \mathcal{F}_t]=\int_EE'\langle H_v(t,e), v-u(t)\rangle\lambda(de),\ \text{dtdP-a.e.}
 \end{aligned}
 \end{equation}
\end{proof}

We now study assumptions, under which the necessary condition (\ref{smp}) becomes a sufficient one.
\begin{theorem}(Sufficient conditions for the optimality of the control)
Let (H5.1) hold and the control $u(\cdot)$\ satisfies (\ref{smp}), where $(p(\cdot), q(\cdot), m(\cdot), n(\cdot,\cdot))$\ is the solution of the adjoint equation (\ref{Eq6.5}). We further assume that the following convexity conditions:\\
(1) $\Phi(x,\widetilde{x})=ax+b\widetilde{x},\ a,b\in\mathbb{R}^n$;\\
(2) $\varphi$ is convex with respect to $x,\widetilde{x}$;\\
(3) $\gamma$ is convex with respect to $y$;\\
(4) Hamiltonian function $H$ is convex with respect to $(x,y,z,k,\widetilde{x},\widetilde{y},\widetilde{z},\widetilde{k},v)$.\\
Then $u$ is an optimal control.
\end{theorem}
\begin{proof}
For any $v\in\mathcal{U}_{ad}$, from (\ref{equ 2018012301}) we have
\begin{equation}\label{equ 2018010701}
\begin{aligned}
J(v(\cdot))-J(u(\cdot))
=E\Big[& \int_0^T\int_EE'\big[g\big(t,\pi^v(t,e),v(t)\big)
-g\big(t,\pi(t,e),u(t)\big)\big]\lambda(de)dt\\
&+E'[\varphi\big(x^v(T),(x^v(T))'\big)-\varphi\big(x(T),(x(T))'\big)]+\gamma\big(y^v(0)\big)-\gamma\big(y(0)\big)\Big].
\end{aligned}
\end{equation}
Since $\varphi$ is convex in $x,\tilde{x}$ and $\gamma$ is convex in $y$, we get
\begin{equation}\label{equ 2018010702}
\begin{aligned}
&\gamma\big(y^v(0)\big)-\gamma\big(y(0)\big)\geq \gamma_y\big(y(0)\big)\big(y^v(0)-y(0)\big),\\
&\varphi\big(x^v(T),(x^v(T))'\big)-\varphi\big(x(T),(x(T))'\big)\\
&\geq \varphi_x\big(x(T),(x(T))'\big)\big(x^v(T)-x(T)\big)+\varphi_{\widetilde{x}}\big(x(T),(x(T))'\big)\big((x^v(T))'-(x(T))'\big).
\end{aligned}
\end{equation}
Observe that $EE'[\varphi_{\widetilde{x}}\big(x(T),(x(T))'\big)\big((x^v(T))'-(x(T))'\big)]=EE'[\varphi_{\widetilde{x}}\big((x(T))',x(T)\big)\big(x^v(T)-x(T)\big)]$,  from (\ref{equ 2018010701}) and (\ref{equ 2018010702}), we obtain
\begin{equation}\label{equ 2018011301}
\begin{aligned}
J(v(\cdot))-J(u(\cdot))\geq& E\Big[ \int_0^T\int_EE'\big[g\big(t,\pi^v(t,e),v(t)\big)
-g\big(t,\pi(t,e),u(t)\big)\big]\lambda(de)dt\\
&+E'[\varphi_x\big(x(T),(x(T))'\big)+\varphi_{\widetilde{x}}\big((x(T))',x(T)\big)]\big(x^v(T)-x(T)\big)+\gamma_y\big(y(0)\big)\big(y^v(0)-y(0)\big)\Big].
\end{aligned}
\end{equation}
Denote $H_x(t,e):=H_x\big(t,x(t),y(t),z(t),k(t,e),(x(t))',(y(t))',(z(t))',(k(t,e))',u(t),p(t),q(t),m(t),n(t,e),e\big)$,
$H_{\widetilde{x}}(t,e)$, $H_{y}(t,e)$, $H_{\widetilde{y}}(t,e),$ $H_{z}(t,e)$, $H_{\widetilde{z}}(t,e),$ $H_{k}(t,e),\ H_{\widetilde{k}}(t,e)$, $H_v(t,e)$ are similarly defined.
Applying It\^o's formula to $q(t)\big(x^v(t)-x(t)\big)$ and taking the expectation, we obtain
\begin{equation}\label{equ 2018010703}
\begin{aligned}
&EE'[\varphi_x\big(x(T),(x(T))'\big)\big(x^v(T)-x(T)\big)+\varphi_{\widetilde{x}}\big((x(T))',x(T)\big)\big(x^v(T)-x(T)\big)]\\
&=E\{E'[a\big(x^v(T)-x(T)\big)+b\big((x^v(T))'-(x(T))'\big)]p(T)\}\\
&-E\int_0^T\int_EE'[\big(x^v(t)-x(t)\big)H_x(t,e)+\big(x^v(t)-x(t)\big)'H_{\widetilde{x}}(t,e)]\lambda(de)dt\\
&+E\int_0^T\int_E\Big[q(t)E'[b(t,\pi^v(t,e),v(t))-b(t,\pi(t,e),u(t))]+m(t)E'[\sigma(t,\pi^v(t,e),v(t))\\
&-\sigma(t,\pi(t,e),u(t))]+n(t,e)E'[h\big(t,\pi^v(t-,e),v(t),e\big)-h\big(t,\pi(t-,e),u(t),e\big)]\Big]\lambda(de)dt.
\end{aligned}
\end{equation}
Applying It\^o's formula to $p(t)\big(y^v(t)-y(t)\big)$ and taking the expectation, we obtain
\begin{equation}\label{equ 2018010901}
\begin{aligned}
&E\big\{p(T)\cdot E'[a\big(x^v(T)-x(T)\big)+b\big((x^v(T))'-(x(T))'\big)]+\gamma_y(y(0))(y^v(0)-y(0))\big\}\\
&=-E\int_0^T\int_EE'\Big[\big(y^v(t)-y(t)\big)H_y(t,e)+\big(y^v(t)-y(t)\big)'H_{\widetilde{y}}(t,e)+\big(z^v(t)-z(t)\big)H_z(t,e)\\
&+\big(z^v(t)-z(t)\big)'H_{\widetilde{z}}(t,e)+\big(k^v(t,e)-k(t,e)\big)H_k(t,e)+\big(k^v(t,e)-k(t,e)\big)'H_{\widetilde{k}}(t,e)\Big]\lambda(de)dt\\
&-E\int_0^T\int_Ep(t)E'[f\big(t,\pi^v(t,e),v(t)\big)-f\big(t,\pi(t,e),u(t)\big)]\lambda(de)dt.
\end{aligned}
\end{equation}
Then, from (\ref{equ 2018011301}), (\ref{equ 2018010703}) and (\ref{equ 2018010901}) we have
\begin{equation}\label{equ 2018011302}
\begin{aligned}
&J(v(\cdot))-J(u(\cdot))\geq -E\int_0^T\int_EE'\Big[\big(x^v(t)-x(t)\big)H_x(t,e)+\big(x^v(t)-x(t)\big)'H_{\widetilde{x}}(t,e)+\big(y^v(t)-y(t)\big)H_y(t,e)\\
&+\big(y^v(t)-y(t)\big)'H_{\widetilde{y}}(t,e)+\big(z^v(t)-z(t)\big)H_z(t,e)+\big(z^v(t)-z(t)\big)'H_{\widetilde{z}}(t,e)+\big(k^v(t,e)-k(t,e)\big)H_k(t,e)\\
&+\big(k^v(t,e)-k(t,e)\big)'H_{\widetilde{k}}(t,e)\Big]\lambda(de)dt
+E\int_0^T\int_EE'\Big[H\big(t,\pi^v(t,e),v(t),p(t),q(t),m(t),n(t,e),e\big)\\
&-H\big(t,\pi(t,e),u(t),p(t),q(t),m(t),n(t,e),e\big)\Big]\lambda(de)dt.
\end{aligned}
\end{equation}
From the convexity of $H$, we know
\begin{equation}\label{equ 2018012302}
\begin{aligned}
&H\big(t,\pi^v(t,e),v(t),p(t),q(t),m(t),n(t,e),e\big)-H\big(t,\pi(t,e),u(t),p(t),q(t),m(t),n(t,e),e\big)\\
&\geq \big(x^v(t)-x(t)\big)H_x(t,e)+\big(y^v(t)-y(t)\big)H_y(t,e)+\big(z^v(t)-z(t)\big)H_z(t,e)+\big(k^v(t,e)-k(t,e)\big)H_k(t,e)\\
&+\big(x^v(t)-x(t)\big)'H_{\widetilde{x}}(t,e)+(y^v(t)-y(t)\big)'H_{\widetilde{y}}(t,e)+(z^v(t)-z(t)\big)'H_{\widetilde{z}}(t,e)
+(k^v(t,e)-k(t,e)\big)'H_{\widetilde{k}}(t,e)\\
&+(v(t)-u(t))H_v(t,e).
\end{aligned}
\end{equation}
From (\ref{equ 2018011302}) and (\ref{equ 2018012302}), we get
\begin{equation}\label{equ 2018012303}
\begin{aligned}
&J(v(\cdot))-J(u(\cdot))\geq E\int_0^T\int_EE'[(v(t)-u(t))H_v(t,e)]\lambda(de)dt.
\end{aligned}
\end{equation}
Combined with the maximum condition (\ref{smp}), we obtain the desired result.
\end{proof}

\section{Applications}
\subsection{Application to mean-variance portfolio selection mixed with a mean-field recursive utility}
In this section, we study a mean-variance portfolio selection mixed with a mean-field recursive utility functional
optimization problem applying the maximum principle derived in Section 5.
 We suppose that there is a financial market consisting of two investment possibilities:\\
(i) a risk-free security (e.g., a bond), where the price $S_0(t)$ at time $t$ is given by
\begin{equation}\label{equ 2018012401}
dS_0(t)=\rho_tS_0(t)dt,\ S_0(0)\geq 0,
\end{equation}
where $\rho_t$ is a bounded deterministic function.\\
(ii) a risky security (e.g., a stock), where the price $S_1(t)$ at time $t$ is given by
\begin{equation}\label{equ 2018012402}
dS_1(t)=S_1(t-)\Big[\mu_tdt+\sigma_tdB_t+\int_E\eta(t,e)\widetilde{\mu}(dedt)\Big],\ S_1(0)>0,
\end{equation}
where $\mu_t\neq 0, \sigma_t\neq 0$, $\eta(t,e)$ are bounded deterministic functions and $\mu_t>\rho_t$.
We also assume that $\eta(t,e)>-1$, for all $t$ and $e\in E$ such that $S_1(t)>0$.\\
Assume that $\theta(t)=(\theta_0(t),\theta_1(t))$ is a portfolio which represents the number of units at time $t$
of the risk-free and the risky security. Then the corresponding wealth process $x(t)$ is given by
\begin{equation}\label{equ 2018012403}
x^\theta(t)=\theta_0(t)S_0(t)+\theta_1(t)S_1(t),\ t\geq 0.
\end{equation}
We also assume the portfolio is self-financing, that is,
\begin{equation}\label{equ 2018012404}
x^\theta(t)=x^\theta(0)+\int_0^t\theta_0(s)dS_0(s)+\int_0^t\theta_1(s-)dS_1(s),\ t\geq 0.
\end{equation}
Let $v(t)=\theta_1(t)S_1(t)$ denote the amount invested in the risky security. Then from (\ref{equ 2018012403}) and (\ref{equ 2018012404}),
we get the wealth dynamics:
\begin{equation}\label{equ 2018012405}
dx^v(t)=[\rho_tx^v(t)+(\mu_t-\rho_t)v(t)]dt+\sigma_tv(t)dB_t+\int_E\eta(t,e)v(t-)\widetilde{\mu}(dtde),
\end{equation}
where $x^v(0)=x_0$ is given.

 We consider
a investor, endowed with initial wealth $x_0>0$, who chooses at each time $t$ his or her portfolio strategy $v(t)$. The investor's object is to
find an admissible portfolio strategy $v(\cdot)\in\mathcal{U}_{ad}$ which maximizes the following expected utility functional:
\begin{equation}\label{equ 2018012407}
{J}(v(\cdot))=E[-\frac{1}{2}(x^v(T)-a)^2]+y^v(t)|_{t=0},
\end{equation}
where
\begin{equation}\label{equ 20180605}
\begin{aligned}
&y^v(t)=E\Big[\gamma x^v(T)+\widetilde{\gamma}E[x^v(T)]\\
&+\int_t^T\alpha\rho_sx^v(s)+\widetilde{\alpha}\rho_sE[x^v(s)]+(\mu_s-\rho_s)v(s)-\beta y^v(s)-\widetilde{\beta} E[y^v(s)]ds|\mathcal{F}_t\Big],\ t\in [0,T],
\end{aligned}
\end{equation}
with nonnegative constants  $a, \gamma, \widetilde{\gamma}, \alpha, \widetilde{\alpha}, \beta, \widetilde{\beta}$.
Notice that the investor's utility functional consists of two parts: One part is the terminal reward
$$E[-\frac{1}{2}(x^v(T)-a)^2];$$
The other part is a mean-field recursive utility functional with generator $f(t,x,\widetilde{x},y,\widetilde{y},v)=\alpha\rho_tx+\widetilde{\alpha}\rho_t\widetilde{x}+(\mu_t-\rho_t)v-
\beta y-\widetilde{\beta}\widetilde{y}$. Mean-field recursive utility is an extension to mean-field (and jumps) of the classical recursive utility concept of Duffie and Epstein \cite{DE1} (i.e., $\widetilde{\alpha}=\widetilde{\beta}=\widetilde{\gamma}\equiv0$ in
(\ref{equ 20180605})), the interested reader can referred to \cite{AHO} and the references therein for the concept of mean-field recursive utility.
\begin{remark}\label{rem6.1}
When only the terminal part is considered for the utility functional,
Framstad, et al. \cite{FOS} solved the above mean-variance portfolio selection by using the sufficient maximum principle in Example 4.1. In addition to the
terminal utility functional,
Shi, Wu \cite{SW} also considered a recursive utility functional for the mean-variance portfolio selection problem. Then, we generalize their recursive utility to  mean-field cases in our model, that is we consider mean-variance portfolio selection mixed with a mean-field recursive utility functional.
\end{remark}
We now apply the result of Section 5 to solve the above optimization problem (\ref{equ 2018012405})-(\ref{equ 2018012407}). In fact, in our jump-diffusion framework, the wealth process $x^v(\cdot)$ in (\ref{equ 2018012405}) and mean-field recursive utility process $y^v(\cdot)$ in (\ref{equ 20180605}) can
be regarded as the solution of the following mean-field FBSDEs with jumps:
\begin{equation}\label{equ 2018012408}
\left\{
\begin{split}
dx^v(t)=&[\rho_tx^v(t)+(\mu_t-\rho_t)v(t)]dt+\sigma_tv(t)dB(t)+\int_E\eta(t,e)v(t-)\widetilde{\mu}(dtde),\\
-dy^v(t)=&\Big[\alpha\rho_tx^v(t)+\widetilde{\alpha}\rho_tE[x^v(t)]+(\mu_t-\rho_t)v(t)-\beta y^v(t)-\widetilde{\beta} E[y^v(t)]\Big]dt\\
&-z^v(t)dB(t)-\int_Ek^v(t,e)\widetilde{\mu}(dtde),\\
x^v(0)=&x_0,\ y^v(T)=\gamma x^v(T)+\widetilde{\gamma}E[x^v(T)],
\end{split}
\right.
\end{equation}
and the optimization problem can be rewritten as
\begin{equation}\label{equ 2018012409}
\mathcal{J}(u(\cdot))=\inf_{v\in\mathcal{U}_{ad}}\mathcal{J}(v(\cdot)),
\end{equation}
where $\mathcal{J}(v(\cdot))=-J(v(\cdot)).$

%20180605 11:19
It is easy to verify that all the assumptions in Section 5 are satisfied for this problem.
The related adjoint equations (\ref{Eq6.5}) become the following form
\begin{equation}\label{equ 2018012410}
\left\{
\begin{split}
dp(t)&=-(\beta+\widetilde{\beta}) p(t)dt,\\
-dq(t)&=\rho_t[q(t)-(\alpha+\widetilde{\alpha})p(t)]dt-m(t)dB(t)-\int_En(t,e)\widetilde{\mu}(dtde),\\
p(0)&=1,\ q(T)=x(T)-a-(\alpha+\widetilde{\alpha}) p(T).
\end{split}
\right.
\end{equation}
%Note that $p(\cdot)$ is deterministic, then the adjoint equation (\ref{equ 2018012410}) can be written as
% \begin{equation}\label{equ 2018012411}
%\left\{
%\begin{split}
%dp(t)&=-\beta p(t)dt,\\
%-dq(t)&=\rho_t[q(t)-p(t)]dt-m(t)dB(t)-\int_En(t,e)\widetilde{\mu}(dtde),\\
%p(0)&=1,\ q(T)=x(T)-a- p(T).
%\end{split}
%\right.
%\end{equation}
Obviously, $p(t)=\exp\{-(\beta+\widetilde{\beta})t\}$, $0\leq t\leq T$. The related Hamiltonian function has the following form
\begin{equation}\label{equ 2018101801}
\begin{aligned}
&H(t,x,y,\widetilde{x},\widetilde{y},v;p,q,m,n,e)\\
=&q[\rho_tx+(\mu_t-\rho_t)v]+m\sigma_tv+n\eta(t,e)v
-p[\alpha\rho_tx+\widetilde{\alpha}\rho_t\widetilde{x}+(\mu_t-\rho_t)v-\beta y-\widetilde{\beta}\widetilde{y}].
\end{aligned}
\end{equation}
Since this is a linear expression of $v$, we get from (\ref{smp})
\begin{equation}\label{equ 2018101802}
\big(q(t)-p(t)\big)(\mu_t-\rho_t)+m(t)\sigma_t+\int_En(t,e)\eta(t,e)\lambda(de)=0.
\end{equation}
We set $q(t)=\varphi_tx(t)+\psi_t$, where $\varphi_t$, $\psi_t$ are deterministic differential functions which will be
specified below. Then, from (\ref{equ 2018012410}) we get
\begin{equation}\label{equ 2018101803}
-\rho_tq(t)+(\alpha+\widetilde{\alpha})\rho_tp(t)=\dot{\varphi_t}x(t)+\dot{\psi_t}+\varphi_t\rho_tx(t)+\varphi_t(\mu_t-\rho_t)u(t),
\end{equation}
and
\begin{equation}\label{equ 2018101804}
\begin{aligned}
m(t)&=\varphi_t\sigma_tu(t),\\
n(t,e)&=\varphi_t\eta(t,e)u(t).
\end{aligned}
\end{equation}
Substituting (\ref{equ 2018101804}) into (\ref{equ 2018101802}), we have
\begin{equation}\label{equ 2018101805}
u(t)=\frac{(\mu_t-\rho_t)(-\varphi_tx(t)-\psi_t+p(t))}{\varphi_t\Lambda_t},
\end{equation}
where $\Lambda_t=\sigma_t^2+\int_E\eta^2(t,e)\lambda(de)$. On the other hand, from (\ref{equ 2018101803}) we get
\begin{equation}\label{equ 2018101806}
u(t)=\frac{-x(t)\dot{\varphi_t}-2\rho_tx(t)\varphi_t-\dot{\psi_t}-\rho_t\psi_t+(\alpha+\widetilde{\alpha})p(t)\rho_t}{\varphi_t(\mu_t-\rho_t)}.
\end{equation}
By comparing (\ref{equ 2018101805}) and (\ref{equ 2018101806}), we obtain the following ordinary differential equation
\begin{equation}\label{equ 2018101807}
\left\{
\begin{aligned}
&\dot{\varphi_t}+\big(2\rho_t-\frac{(\mu_t-\rho_t)^2}{\Lambda_t}\big)\varphi_t=0,\ \  \varphi_T=1,\\
&\dot{\psi_t}+\big(\rho_t-\frac{(\mu_t-\rho_t)^2}{\Lambda_t}\big)\psi_t-(\alpha+\widetilde{\alpha})\rho_t^2+\frac{(\mu_t-\rho_t)^2}{\Lambda_t}p(t)=0,\ \
\psi_T=-a-(\alpha+\widetilde{\alpha})p(T).
\end{aligned}
\right.
\end{equation}
Then we obtain
\begin{equation}\label{equ 2018101808}
\varphi_t=\exp\{\int_t^T\big(2\rho_s-\frac{(\mu_s-\rho_s)^2}{\Lambda_s}\big)ds\},\ 0\leq t\leq T,
\end{equation}
and
\begin{equation}\label{equ 2018101809}
\begin{aligned}
\psi_t=&[-a-(\alpha+\widetilde{\alpha})p(T)]\exp\{\int_t^T\big(\rho_s-\frac{(\mu_s-\rho_s)^2}{\Lambda_s}\big)ds\}
-\int_t^T\Big[\big(\frac{(\mu_s-\rho_s)^2}{\Lambda_s}p(s)-(\alpha+\widetilde{\alpha})\rho_s^2\big)\\
&\exp\{\int_s^T(\frac{(\mu_r-\rho_r)^2}{\Lambda_r}-\rho_r)dr\}\Big]ds\cdot
\exp\{\int_t^T(\rho_s-\frac{(\mu_s-\rho_s)^2}{\Lambda_s})ds\}
,\ 0\leq t\leq T.
\end{aligned}
\end{equation}
Finally, by combining the above discussion and Theorem 5.2, we obtain the following theorem.
\begin{theorem}
The optimal solution $u$ of our mean-variance portfolio selection mixed with a mean-field recursive utility (\ref{equ 2018012405}) and (\ref{equ 2018012407}) is given (in feedback form) by
\begin{equation}\label{equ 2018101810}
u(t)=\frac{(\mu_t-\rho_t)(-\varphi_tx(t)-\psi_t+p(t))}{\varphi_t\Lambda_t},
\end{equation}
where $\Lambda_t=\sigma_t^2+\int_E\eta^2(t,e)\lambda(de)$, $p(t)=\exp\{-(\beta+\widetilde{\beta})t\}$, $\varphi_t$ and $\psi_t$ are given by (\ref{equ 2018101808}) and (\ref{equ 2018101809}), respectively.
\end{theorem}

\subsection{Application to linear-quadratic optimal control problem}
Now we consider an example of linear-quadratic stochastic control problem. The dynamic of our problem is the following linear
mean-field FBSDEs with jumps
\begin{equation}\label{eg2.1}
\left\{
\begin{aligned}
dx^v(t)=&\ \{ax^v(t)+\widetilde{a}E[x^v(t)]\}dt+\{bx^v(t)+Bv(t)\}dB_t+\int_E\{L(e)v(t)\}\widetilde{\mu}(dtde),\\
-dy^v(t)=&\ \{cx^v(t)+\widetilde{c}E[x^v(t)]+ly^v(t)+\widetilde{l}E[y^v(t)]+Dv(t)\}dt-z^v(t)dB_t-\int_Ek^v(t,e)\widetilde{\mu}(dtde),\\
x^v(0)=&\ a,\ \ y^v(T)=x^v(T),
\end{aligned}
\right.
\end{equation}
where $a,\widetilde{a},b,B,c,\widetilde{c},l,\widetilde{l},D$ are constants, $L(e)$ is bounded deterministic function and $v\in\mathcal{U}_{ad}$.

The cost functional is a quadratic one, and it has the form
\begin{equation}\label{eg2.2}
\begin{aligned}
 J(v(\cdot))=\frac{1}{2}\int_0^TRE[x^v(t)]^2dt+\frac{1}{2}NE[x^v(T)]^2+\frac{1}{2}QE[y^v(0)]^2,
\end{aligned}
\end{equation}
where $R,N,Q$ are positive constants. Then the related Hamiltonian function has the following form
\begin{equation}\label{eg2.3}
H(x,\widetilde{x},y,\widetilde{y},v,p,q,m,n,e)=q(ax+\widetilde{a}\widetilde{x})+m(bx+Bv)+nL(e)v-p(cx+\widetilde{c}\widetilde{x}+ly+\widetilde{l}\widetilde{y}
+Dv)+\frac{1}{2}Rx^2.
\end{equation}
The adjoint equation can be written as
\begin{equation}\label{eg2.4}
\left\{
\begin{aligned}
dp(t)=&\ (l+\widetilde{l})p(t)dt,\\
-dq(t)=&\ \{aq(t)+bm(t)-cp(t)+Rx(t)+\widetilde{a}E[q(t)]-\widetilde{c}p(t)\}dt-m(t)dB_t-\int_En(t,e)\widetilde{\mu}(dtde),\\
p(0)=&\ -Qy(0),\ \ q(T)=Nx(T)-p(T).
\end{aligned}
\right.
\end{equation}
Then, $p(t)=-Qy(0)\exp{(l+\widetilde{l})t}$, $t\in[0,T]$.\\
From (\ref{smp}), we have
\begin{equation}\label{eg2.5}
Bm(t)+\int_En(t,e)L(e)\lambda(de)-p(t)D=0.
\end{equation}
We assume
\begin{equation}\label{eg2.6}
q(t)=\phi(t)x(t)+\psi(t)E[x(t)]+\theta(t),
\end{equation}
where $\phi(t)$, $\psi(t)$, $\theta(t)$ are deterministic differentiable functions.
Applying It\^o's formula to (\ref{eg2.6}), from (\ref{eg2.1}) we have
\begin{equation}\label{eg2.7}
\begin{aligned}
dq(t)=&\ \Big\{\big[\dot{\phi}(t)+a\phi(t)\big]x(t)+\big[\widetilde{a}\phi(t)+\dot{\psi}(t)+(a+\widetilde{a})\psi(t)\big]E[x(t)]+\dot{\theta}(t)\Big\}dt\\
&\ +\phi(t)[bx(t)+Bu(t)]dB_t+\int_E\phi(t)L(e)u(t)\widetilde{\mu}(dtde).
\end{aligned}
\end{equation}
Compared with (\ref{eg2.4}), we obtain
\begin{equation}\label{eg2.8}
m(t)=\phi(t)[bx(t)+Bu(t)],
\end{equation}
\begin{equation}\label{eg2.9}
n(t,e)=\phi(t)L(e)u(t),
\end{equation}
\begin{equation}\label{eg2.10}
-aq(t)-bm(t)+(c+\widetilde{c})p(t)-Rx(t)-\widetilde{a}E[q(t)]=\big[\dot{\phi}(t)+a\phi(t)\big]x(t)+\big[\widetilde{a}\phi(t)+\dot{\psi}(t)+(a+\widetilde{a})\psi(t)\big]E[x(t)]+\dot{\theta}(t)
\end{equation}
Substituting (\ref{eg2.8}), (\ref{eg2.9}) into (\ref{eg2.5}), we get
\begin{equation}\label{eg2.11}
\phi(t)u(t)=\frac{1}{\Lambda}[p(t)D-Bb\phi(t)x(t)],
\end{equation}
where $\Lambda=B^2+\int_EL^2(e)\lambda(de)$. Then from (\ref{eg2.10}) and (\ref{eg2.11}), we have
\begin{equation}\label{eg2.12}
\begin{aligned}
\Big[\dot{\phi}(t)+(2a+b^2-\frac{B^2b^2}{\Lambda})\phi(t)+R\Big]x(t)
+\Big[\dot{\psi}(t)+(2a+2\widetilde{a})\psi(t)+2\widetilde{a}\phi(t)\Big]E[x(t)]\\
+\dot{\theta}(t)+(a+\widetilde{a})\theta(t)+\frac{bBD}{\Lambda}p(t)-(c+\widetilde{c})p(t)=0.
\end{aligned}
\end{equation}
Noting the terminal condition in (\ref{eg2.4}), we get
\begin{equation}\label{eg2.13}
\begin{aligned}
&\dot{\phi}(t)+(2a+b^2-\frac{B^2b^2}{\Lambda})\phi(t)+R=0,\ \phi(T)=N,\\
&\dot{\psi}(t)+(2a+2\widetilde{a})\psi(t)+2\widetilde{a}\phi(t)=0,\ \psi(T)=0,\\
&\dot{\theta}(t)+(a+\widetilde{a})\theta(t)+\frac{bBD}{\Lambda}p(t)-(c+\widetilde{c})p(t)=0,\ \theta(T)=-p(T).
\end{aligned}
\end{equation}
The solutions of these equations are
\begin{equation}\label{eg2.14}
\begin{aligned}
&\phi(t)=[N+\frac{R}{2a^2+b^2-\frac{b^2B^2}{\Lambda}}]\exp\{(2a^2+b^2-\frac{b^2B^2}{\Lambda})(T-t)\}-\frac{R}{2a^2+b^2-\frac{b^2B^2}{\Lambda}},\\
&\psi(t)=\exp\{-2(a+\widetilde{a})t\}\int_t^T2\widetilde{a}\phi(s)\exp\{2(a+\widetilde{a})s\}ds,\\
&\theta(t)=-p(T)exp\{(a+\widetilde{a})(T-t)\}-\int_t^T[c+\widetilde{c}-\frac{bBD}{\Lambda}]p(s)\exp\{(a+\widetilde{a})s\}ds.
\end{aligned}
\end{equation}
Finally, since the assumptions of Theorem 5.2 are satisfied in our case, we get the following result.
\begin{theorem}
The optimal solution $u$ of our linear-quadratic control problem (\ref{eg2.1}) and (\ref{eg2.2}) is given (in feedback form) by
\begin{equation}\label{eg2.11}
u(t)=\frac{1}{\Lambda\phi(t)}[p(t)D-Bb\phi(t)x(t)],
\end{equation}
where $\Lambda=B^2+\int_EL^2(e)\lambda(de)$,
 $p(t)=-Qy(0)\exp{(l+\widetilde{l})t}$, $\phi(t)$ is given by (\ref{eg2.14}).
\end{theorem}

\section{{ Appendix}}
\begin{lemma}\label{Lem3.1}
Assume (H3.1) and (H3.2) hold. If for an $\alpha_0\in[0,1)$ there exists a solution $(x^{\alpha_0}, y^{\alpha_0}, z^{\alpha_0},k^{\alpha_0})$ of
equation (\ref{Eq3.2}), then there exists a positive constant $\delta_0$ such that, for each $\delta\in[0, \delta_0]$ there exists a solution
$(x^{\alpha_0+\delta}, y^{\alpha_0+\delta}, z^{\alpha_0+\delta}, k^{\alpha_0+\delta})$\ of mean-field FBSDE with jumps (\ref{Eq3.2}) for $\alpha=\alpha_0+\delta$.
\end{lemma}
\begin{proof} Since there
exists a (unique) solution of equation (\ref{Eq3.2}) for every $\phi\in \mathcal{H}_{\bar{\mathbb{F}}}^2(0,T;\mathbb{R}^n),\ \gamma\in \mathcal{H}_{\bar{\mathbb{F}}}^2(0,T;\mathbb{R}^m),$ $\psi\in \mathcal{H}_{\bar{\mathbb{F}}}^2(0,T,\mathbb{R}^{n\times d}),\ \varphi\in\mathcal{K}^2_{\bar{\mathbb{F}},\lambda}(0,T;\mathbb{R}^n),\  \alpha_0\in[0,1)$, then for each $x(T)\in L^2(\Omega, \mathcal{F}_T, P)$ and a quadruple $(\lambda(t,e))_{0\leq t\leq T}=(x(t),y(t),z(t),k(t,e))_{0\leq t\leq T}\in \mathcal{H}_{\mathbb{F}}^2(0,T;\mathbb{R}^{n+m+m \times d})\times\mathcal{K}^2_{\mathbb{F},\lambda}(0,T;\mathbb{R}^m)$ and $\delta>0$,
the following mean-field FBSDE with jumps
\begin{equation}
\left\{
\begin{aligned}
dX(t)=&\ \[\alpha_0\int_E E'[b(t,\Lambda(t,e),(\Lambda(t,e))')]\lambda(de)+\delta \int_EE'[b(t,\lambda(t,e),(\lambda(t,e))')]\lambda(de)+E'[\phi(t)]\]dt\\
      &\  +\[\alpha_0\int_E E'[\sigma(t,\Lambda(t,e),(\Lambda(t,e))')]\lambda(de)+\delta \int_E E'[\sigma(t,\lambda(t,e),(\lambda(t,e))')]\lambda(de)+E'[\psi(t)]\]dB_t\\
      &\ +\int_E \Big[\alpha_0 E'[h(t,\Lambda(t,e),(\Lambda(t,e))',e)]+\delta E'[h(t,\lambda(t,e),(\lambda(t,e))',e)]+E'[\varphi(t,e)]\Big]\widetilde{\mu}(dt,de),\\
-dY(t)=&\ \[(1-\alpha_0)\beta_1GX(t)+\alpha_0\int_EE'[f(t,\Lambda(t,e),(\Lambda(t,e))')]\lambda(de)+\delta\big(-\beta_1Gx(t)\\
       &\ \ \ + \int_EE'[f(t,\lambda(t,e),(\lambda(t,e))')]\lambda(de)\big)+E'[\gamma(t)]\]dt -Z(t)dB_t-\int_EK(t,e)\widetilde{\mu}(dt,de),\\
X(0)=&\ a,\\
Y(T)=&\ \alpha_0 E'[\Phi(X(T),(X(T))')]+(1-\alpha_0)GX(T) +\delta\(E'[\Phi(x(T),(x(T))')]-Gx(T)\)+\xi,
\end{aligned}
\right.
\end{equation}
 exists
a unique solution
$$(\Lambda(t,e))_{0\leq t\leq T}=(X(t),Y(t),Z(t),K(t,e))_{0\leq t\leq T}\in \mathcal{H}_{\mathbb{F}}^2(0,T;\mathbb{R}^{n+m+{m \times d}})\times\mathcal{K}^2_{\mathbb{F},\lambda}(0,T;\mathbb{R}^m).$$
We now prove that  the mapping $I_{\alpha_0+\delta}$ defined by
$$I_{\alpha_0+\delta}(\lambda\times x(T))=\Lambda \times X(T):$$
$$\mathcal{H}_{\mathbb{F}}^2(0,T;\mathbb{R}^{n+m+m \times d})\times\mathcal{K}^2_{\mathbb{F},\lambda}(0,T;\mathbb{R}^m)\times L^2(\Omega,\mathcal{F}_T,P)\mapsto \mathcal{H}_{\mathbb{F}}^2(0,T;\mathbb{R}^{n+m+m \times d})\times\mathcal{K}^2_{\mathbb{F},\lambda}(0,T;\mathbb{R}^m)\times L^2(\Omega,\mathcal{F}_T,P)$$
 is a contraction when $\delta$ is small enough. For any $\bar{\lambda}=(\bar{x},\bar{y},\bar{z},\bar{k})\in \mathcal{H}_{\mathbb{F}}^2(0,T;\mathbb{R}^{n+m+m \times d})\times\mathcal{K}^2_{\mathbb{F},\lambda}(0,T;\mathbb{R}^m)$ and $\bar{x}(T)\in L^2(\Omega, \mathcal{F}_T, P)$, we denote
$$\bar{\Lambda}\times \bar{X}(T)=I_{\alpha_0+\delta}(\bar{\lambda}\times \bar{x}(T)),\
\widehat{\lambda}=(\widehat{x},\widehat{y},\widehat{z},\widehat{k})=(x-\bar{x},y-\bar{y},z-\bar{z},k-\bar{k}),$$ %$$\widehat{\lambda}'=(\widehat{x}',\widehat{y}',\widehat{z}',\widehat{k}')=(x'-\bar{x}',y'-\bar{y}',z'-\bar{z}',
%k-\bar{k}'),\ $$
$$\widehat{\Lambda}=(\widehat{X},\widehat{Y},\widehat{Z},\widehat{K})=(X-\bar{X},Y-\bar{Y},Z-\bar{Z},K-\bar{K}). $$ %$$\widehat{\Lambda}'=(\widehat{X}',\widehat{Y}',\widehat{Z}',\widehat{K}')=(X'-\bar{X}',Y'-\bar{Y}',Z'-\bar{Z}',K'-\bar{K}').$$
Applying It\^{o}'s formula to $\langle G\widehat{X}(t),\widehat{Y}(t)\rangle$ it yields
 \begin{equation}\nonumber
  \begin{aligned}
 &\ E \Big\langle \alpha_0E'[\Phi(X(T),(X(T))')-\Phi(\bar{X}(T),(\bar{X}(T))')]+(1-\alpha_0)G\widehat{X}(T)\\
&\ +\delta\Big(E'[\Phi(x(T),(x(T))')-\Phi(\bar{x}(T),(\bar{x}(T))')]-G\widehat{x}(T)\Big) ,G\widehat{X}(T) \Big\rangle \\
=&\ E \int_0^T\int_E \big\langle\alpha_0E'[A(t,\Lambda(t,e),(\Lambda(t,e))')-A(t,\bar{\Lambda}(t,e),(\bar{\Lambda}(t,e))')],
\widehat{\Lambda}(t,e) \big\rangle\lambda(de)dt\\
&\ -E\int_0^T(1-\alpha_0)\beta_1\langle G\widehat{X}(t),G\widehat{X}(t)\rangle dt +E\int_0^T\delta \beta_1\langle G\widehat{X}(t),G\widehat{x}(t)\rangle dt\\
&\ + E\int_0^T\int_E\delta E'\[\langle G\widehat{X}(t),\widehat{f}(t,e)\rangle
+\langle G^T\widehat{Y}(t),\widehat{b}(t,e)\rangle+\langle G^T\widehat{Z}(t),\widehat{\sigma}(t,e)\rangle+\langle G^T\widehat{K}(t,e),\widehat{h}(t,e)\rangle\]\lambda(de) dt,
\end{aligned}
\end{equation}
where
\begin{equation}\nonumber
\begin{aligned}
\widehat{b}(t,e)=&\ b(t,\lambda(t,e),(\lambda(t,e))')-b(t,\bar{\lambda}(t,e),(\bar{\lambda}(t,e))'),\\
\widehat{\sigma}(t,e)=&\ \sigma(t, \lambda(t,e), (\lambda(t,e))')-\sigma(t, \bar{\lambda}(t,e), (\bar{\lambda}(t,e))'),\\
\widehat{h}(t,e)=&\ h(t, \lambda(t,e), (\lambda(t,e))',e)-h(t, \bar{\lambda}(t,e), (\bar{\lambda}(t,e))',e),\\
\widehat{f}(t,e)=&\ -f(t,\lambda(t,e),(\lambda(t,e))')+f(t,\bar{\lambda}(t,e), (\bar{\lambda}(t,e))').
\end{aligned}
\end{equation}
From the assumptions (H3.1) and (H3.2), we know\\
(1) if $\beta_1-C_0L_A=0,\ \mu_1-L_\Phi\lambda_1>0$, $\beta_2-C_0L_A>0$, $\beta_3-C_0L_A>0$, then we have
\begin{equation}\label{equ 2018110101}
\begin{aligned}
 &\ E\[\int_0^T (|\widehat{Y}(t)|^2+|\widehat{Z}(t)|^2+\int_E|\widehat{K}(t,e)|^2\lambda(de)) dt\]\\
 \leq&\ \delta C_2E\Big\{\int_0^T\(|\widehat{X}(t)|^2+|\widehat{Y}(t)|^2+|\widehat{Z}(t)|^2+\int_E|\widehat{K}(t,e)|^2\lambda(de)\)dt+|\widehat{X}(T)|^2
 +|\widehat{x}(T)|^2\\
&\ + \int_0^T\(|\widehat{x}(t)|^2+|\widehat{y}(t)|^2+|\widehat{z}(t)|^2+\int_E|\widehat{k}(t,e)|^2\lambda(de)\)dt\Big\}.
\end{aligned}
\end{equation}
On the other hand, from standard technique to the forward equation for $\widehat{X}(t)=X(t)-\bar{X}(t)$, we get
 \begin{equation}\label{equ 2018110102}
  \begin{aligned}
\sup_{0\leq t\leq T}E[|\widehat{X}(t)|^2] &\ \leq \delta C_2E\[\int_0^T\(|\widehat{x}(t)|^2+|\widehat{y}(t)|^2+|\widehat{z}(t)|^2+\int_E|\widehat{k}(t,e)|^2\lambda(de)\)dt\]\\
 &\ +C_2E\[\int_0^T\(|\widehat{Y}(t)|^2+|\widehat{Z}(t)|^2+\int_E|\widehat{K}(t,e)|^2\lambda(de)\)dt\].
\end{aligned}
\end{equation}
From (\ref{equ 2018110101}) and (\ref{equ 2018110102}) we get
\begin{equation}\label{equ 2018110103}
\begin{aligned}
&\ E\Big\{\int_0^T\(|\widehat{X}(t)|^2+|\widehat{Y}(t)|^2+|\widehat{Z}(t)|^2+\int_E|\widehat{K}(t,e)|^2\lambda(de)\)dt
+|\widehat{X}(T)|^2\Big\}\\
\leq&\ \bar{C}\delta E\Big\{\int_0^T\(|\widehat{x}(t)|+|\widehat{y}(t)|^2+|\widehat{z}(t)|^2+\int_E|\widehat{k}(t,e)|^2\lambda(de)\)dt
+|\widehat{x}(T)|^2\Big\}.
\end{aligned}
\end{equation}
Here the constant $\bar{C}$ depends on the Lipschitz constants, $\lambda_1,\ \beta_1,\ \beta_2,\ \beta_3,\ C_0$ and $T$.  \\
(2) If $\beta_1-C_0L_A>0$, $\beta_2-C_0L_A\geq0$, $\beta_3-L_A\geq0$, $\mu_1-L_\Phi\lambda_1>0$, then we have
  \begin{equation}\label{equ 2018110104}
  \begin{aligned}
 &\ E[|\widehat{X}(T)|^2]+E[\int_0^T |\widehat{X}(t)|^2 dt]\\
 \leq&\ \delta C_1E\Big\{\int_0^T\(|\widehat{X}(t)|^2+|\widehat{Y}(t)|^2+|\widehat{Z}(t)|^2+
 \int_E|\widehat(K)(t,e)|^2\lambda(de)\)dt
 +|\widehat{X}(T)|^2\\
&\ +\int_0^T\(|\widehat{x}(t)|^2+|\widehat{y}(t)|^2+|\widehat{z}(t)|^2+\int_E|\widehat{k}(t,e)|^2\lambda(de)\)dt
+|\widehat{x}(T)|^2\Big\}.
  \end{aligned}
  \end{equation}
From the standard estimate of the mean-field BSDE part, we get
\begin{equation}\label{equ 2018110105}
\begin{aligned}
&\ E\[\int_0^T\(|\widehat{Y}(t)|^2+|\widehat{Z}(t)|^2+\int_E|\widehat{K}(t,e)|^2\lambda(de)\)dt\]\\
\leq&\ C_1\delta E\Big\{\int_0^T\(|\widehat{x}(t)|^2+|\widehat{y}(t)|^2+|\widehat{z}(t)|^2+\int_E|\widehat{k}(t,e)|^2\lambda(de)\)dt
+|\widehat{x}(T)|^2\Big\}+C_1\Big\{E\int_0^T|\widehat{X}(t)|^2dt+E|\widehat{X}(T)|^2\Big\}.
\end{aligned}
\end{equation}
Here the constant $C_1$ depends on the Lipschitz constants, $\lambda_1,\ \beta_1,\ \mu_1,\ C_0,\ \alpha_0$, and $T$.\\
From (\ref{equ 2018110104}), (\ref{equ 2018110105}) and the standard estimate of $\widehat{X}(t)$, it follows that, for the sufficiently small $\delta>0$,
\begin{equation}\label{equ 2018110106}
\begin{aligned}
&\ E\Big\{\int_0^T\(|\widehat{X}(t)|^2+|\widehat{Y}(t)|^2+|\widehat{Z}(t)|^2+\int_E|\widehat{K}(t,e)|^2\lambda(de)\)dt
+|\widehat{X}(T)|^2\Big\}\\
\leq&\ \bar{C}\delta E\Big\{\int_0^T\(|\widehat{x}(t)|^2+|\widehat{y}(t)|^2+|\widehat{z}(t)|^2+\int_E|\widehat{k}(t,e)|^2\lambda(de)\)dt
+|\widehat{x}(T)|^2\Big\}.
\end{aligned}
\end{equation}
Here the constant $\bar{C}$ depends only on the Lipschitz constants, $\lambda_1,\ \beta_1,\ \mu_1,\ \alpha_0$ and $T$. \\
From above all, we now choose $\delta_0=\frac{1}{2\bar{C}}$ in (\ref{equ 2018110103}) and (\ref{equ 2018110106}). Obviously, for every fixed $\delta\in[0,\delta_0]$,
the mapping $I_{\alpha_0+\delta}$ is a contraction in the sense that
\begin{equation}\nonumber
\begin{aligned}
&\ E\Big\{\int_0^T\(|\widehat{X}(t)|^2+|\widehat{Y}(t)|^2+|\widehat{Z}(t)|^2+\int_E|\widehat{K}(t,e)|^2\lambda(de)\)dt
+|\widehat{X}(T)|^2\Big\}\\
\leq&\ \frac{1}{2} E\Big\{\int_0^T\(|\widehat{x}(t)|^2+|\widehat{y}(t)|^2+|\widehat{z}(t)|^2+\int_E|\widehat{k}(t,e)|^2\lambda(de)\)dt
+|\widehat{x}(T)|^2\Big\}.
\end{aligned}
\end{equation}
It means immediately that this mapping has a unique fixed point
$$\Lambda^{\alpha_0+\delta}=(X^{\alpha_0+\delta},Y^{\alpha_0+\delta},Z^{\alpha_0+\delta},K^{\alpha_0+\delta}),$$
which is the solution of equation (\ref{Eq3.2}) for $\alpha=\alpha_0+\delta$.
\end{proof}

\end{document}